\documentclass{article} 
\usepackage{basic}
\usepackage{float}
\usepackage[nottoc,numbib]{tocbibind}
\title{Computing fusion products of MV cycles using the Mirkovi\'c--Vybornov isomorphism}
\author{Roger Bai, Anne Dranowski, Joel Kamnitzer} 
\date{\today}

\begin{document}

\maketitle
\begin{abstract}
    The fusion of two Mirkovi\'c--Vilonen cycles is a degeneration of their product, defined using the Beilinson--Drinfeld Grassmannian.  In this paper, we put in place a conceptually elementary approach to computing this product in type \(A\).  We do so by transferring the problem to a fusion of generalized orbital varieties using the Mirkovi\'c--Vybornov isomorphism.  As an application, we explicitly compute all cluster exchange relations in the coordinate ring of the upper-triangular subgroup of \( \GL_4\), confirming that all the cluster variables are contained in the Mirkovi\'c--Vilonen basis.
\end{abstract}

\section{Introduction}\label{s:intro}
\subsection{Geometric Satake and Mirkovi\'c--Vilonen cycles}\label{ss:gsmv}
The geometric Satake equivalence of Mirkovi\'c and Vilonen~\cite{mirkovic2007geometric} is an equivalence of categories between the category of spherical perverse sheaves on the affine Grassmannian of a reductive group \( G \) and the category of representations of its Langlands dual group \(G^\vee \).  This foundational result provides a powerful tool to study representation theory using geometry.  In particular, under this equivalence, the MV cycles (short for Mirkovi\'c--Vilonen) in the affine Grassmannian of \( G \) index bases for irreducible representations of \( G^\vee \). In this paper we work with \( G = \GL_m\) and identify \( G^\vee = \GL_m\) as well.

In~\cite{mirkovic2007quiver} (see also the recent sequel~\cite{mirkovic2019comparison}) Mirkovi\'c and Vybornov supply a geometric version of symmetric and skew Howe \((\GL_m,\GL_N)\) dualities.  They relate Kazhdan--Lusztig slices in the affine Grassmannian of \(\GL_m\) to slices in \(N\times N\) nilpotent orbits on the one hand, and to Nakajima quiver varieties on the other hand.    The second author~\cite{dthesis} showed  that under the Mirkovi\'c--Vybornov isomorphism, the MV cycles are identified with certain varieties of matrices, called generalized orbital varieties.

Essential to the geometric Satake equivalence, the BD Grassmannian (short for Beilinson--Drinfeld) is used to define a fusion product on the category of spherical perverse sheaves on the affine Grassmannian.  In~\cite{anderson2003polytope}, Anderson used the BD Grassmannian to define a fusion of MV cycles.  He constructed a family where the general fibre is a Cartesian product of two MV cycles, and the special fibre, called the fusion, is a (non-reduced) union of MV cycles.  Anderson conjectured that this fusion product matches the multiplication in the coordinate ring of the unipotent subgroup \( N \subset\GL_m\).  This conjecture was proven in~\cite{baumann2019mirkovic}, where it was also established that the MV cycles give a basis for \( \CC[N]\).

\subsection{Fusion using generalized orbital varieties}\label{ss:fugov}
Computing the fusion product of MV cycles using the BD Grassmannian is quite difficult, both computationally and conceptually.  A few such computations were done by Anderson and Kogan in~\cite{anderson2006algebra} in a somewhat ad-hoc fashion.  Another perspective on these computations (using convolution Grassmannians) was recently studied by Baumann, Gaussent and Littelmann~\cite{baumann2020bases}.

The purpose of this paper is to give a conceptually elementary way to compute this product by transferring it to a fusion product of generalized orbital varieties.  To do so, we use that the Mirkovi\'c--Vybornov isomorphism extends to an isomorphism between families of slices in the BD Grassmannian and families of slices of matrices with two eigenvalues. Our main result is the following (see \Cref{cor:goodreps} for a more precise statement).

\begin{theorem}\label{thm:main}
    Let \( Z', Z'' \) be MV cycles. The structure constants for the of the corresponding MV basis vectors in \(\CC[N]\) are equal to the intersection multiplicities of generalized orbital varieties in the fusion of the generalized orbital varieties corresponding to \( Z', Z''\).
\end{theorem}

Since these generalized orbital varieties are subschemes of affine spaces of matrices defined by certain rank conditions, the fusion of generalized orbital varieties can be computed by elementary commutative algebra.

\subsection{Cluster algebras and MV cycles}\label{ss:clmv}
One motivation for computing products of MV basis elements is to relate the cluster algebra structure on \(\CC[N]\) to the MV basis. Cluster algebras were introduced by Fomin and Zelevinsky in~\cite{fomin2002cluster1} and it was shown in~\cite{geiss2007initial}, using a construction of~\cite{berenstein2005cluster3}, that \(\CC[N]\) has a natural cluster algebra structure. 

This structure starts with an initial seed \(\Sigma = (\{x_1,\dots,x_r\},B)\), where \(x_i \in \CC[N]\) are a cluster of variables, \(B\) is a certain matrix, and \( r = m(m-1)/2 \).
We obtain new variables through a process called mutation, where for a mutable variable \(x_i\), its mutation is the unique element \(x_i^* \in \CC[N]\) such that 
\[
    x_i x_i^* = x_+ + x_- \,.
\]
This equation is called an exchange relation, and the terms \(x_+\) and \(x_-\) are certain monomials in the \(x_j\), \(j\neq i\),  determined by \(B\). 
It gives us the new cluster \(\{x_1, \dots, x_i^*, \dots, x_r\}\). 

The initial cluster \(\{x_1,\dots,x_r\}\) as well as all possible variables that are obtained through successive mutations are known as cluster variables. 
The cluster monomials are products of cluster variables that are supported on a single cluster. 

There are two well-known bases of \(\CC[N]\) whose relationship with the cluster structure has been heavily studied: the dual canonical basis of~\cite{lusztig1990canonicalbases} and the dual semicanonical basis of~\cite{lusztig2000semicanonical}.
In~\cite{kang2018monoidal}, Kang, Kashiwara, Kim and Oh showed that the cluster monomials are contained in the dual canonical basis, while, in~\cite{geiss2006rigid}, Gei\ss, Leclerc and Schr\"oer proved that they are all contained in the dual semicanonical basis.
%

The MV basis, the dual canonical basis, and the dual semicanonical basis are all examples of biperfect bases of \(\CC[N]\)~\cite[Section 2]{baumann2019mirkovic}. In particular, they share a crystal structure.
In~\cite[Appendix]{baumann2019mirkovic}, we showed that the MV basis differs from the dual semicanonical basis at the same point in the underlying crystal as where the dual semicanonical basis differs from the dual canonical basis.  
The fact that this point at which each of these bases differ is not a cluster monomial in any of these bases, gives evidence for the following conjecture of Anderson and Kogan~\cite[Conjecture 5.1]{anderson2006algebra}.
\begin{conjecture}\label{conj:cluster in MV}
    The cluster monomials are contained in the MV basis. 
\end{conjecture}

The conjecture is true in type \(A_n\) for \(n\leq 3\) since in these cases \(\CC[N]\) has a unique biperfect basis \cite[Section 2.3]{baumann2019mirkovic}.
More generally, cluster variables called flag minors, coming from initial seeds constructed in \cite[Definitions~2.2--2.3]{berenstein2005cluster3}, belong to each biperfect basis of \(\CC[N]\) \cite[Remark~2.10]{baumann2019mirkovic}. 
And recently in \cite[Proposition 7.2]{baumann2020bases} cluster monomials supported on clusters of flag minors satisfying a certain combinatorial condition where shown to be in the MV basis. 

Our hope is that each exchange relation can be realized using the fusion of MV cycles. This would imply that all cluster variables are in the MV basis. In \Cref{ss:GL4 examples} we verify that this is indeed the case in type \(A_3\) by directly computing each fusion product corresponding to an exchange relation.

\section{Supporting Cast}\label{s:players}
\subsection{Rings and discs}\label{ss:rings}
Set \(\cO = \cO_0 = \CC\xt\) and \(\cK = \cK_0 = \CC\xT\).
We will also consider \(\cO_s = \CC\xt[t-s]\) and its fraction field \(\cK_s=\CC\xT[t-s]\), for any \( s \in \CC\setminus\{0\} \), as well as \( \Oinf = \CC\xt[t^{-1}] \) and \(\Kinf = \CC\xT[t^{-1}]\). For any point \( s \in \PP = \PP^1\), \( \cO_s\) is the completion of the local ring \( \mathcal O_{\PP, s} \) and thus the formal spectrum of \( \cO_s\) is the formal neighbourhood \( D_s\) of \( s \) (also called the formal disc centered at \( s\)).  Similarly, the formal spectrum of the field \(\cK_s\) is the deleted formal neighbourhood (or punctured disc), denoted \( D_s^\times\).

Note that for \( s \in \AA = \AA^1\), we have an obvious isomorphism  \(\cO_s \cong \cO\) and \(\cK_s\cong\cK\) taking \( t-s \) to \( t\). 

\subsection{Groups}\label{ss:groups}
Let \(H \) be an algebraic group over \( \CC \).  We will be interested in \( H(R)\) where \( R \) is a \(\CC\)-algebra, for example \(R = \CC[t], \cO_s\), etc.  
Note that evaluation at \( t = s\) provides a group homomorphism \( H(\cO_s) \rightarrow H\).  We denote the kernel of this map by \( H_1(\cO_s)\), often called the  first congruence subgroup.  We will be particularly interested in this construction in the case \( s = \infty\), which gives us the group \( H_1(\Oinf)\).

Throughout the paper, we fix \( m \in \NN\).  We let \(G = \GL_m\) and we let \(T\subset G\) be the maximal torus of diagonal matrices. Let \( N, N_- \subset G\) denote the subgroups of upper and lower triangular matrices with 1s on the diagonal. 
We identify \(\ZZ^m\) with the coweight lattice of \(G\) and say that a coweight \(\nu = (\nu_1,\dots,\nu_m)\) is \new{dominant} if \(\nu_1\ge\cdots\ge\nu_m\) and \new{effective} if \(\nu_j\ge 0\) for all \(j\).
If \(\nu\) is both effective and dominant, then it is a partition of size \(|\nu| = \nu_1 + \cdots + \nu_m\in\NN\).  Let \( Q_+ \subset \ZZ^m\) denote the positive root cone, explicitly we have
\[
    Q_+ = \{ (\nu_1, \dots, \nu_m) : \nu_1 + \cdots + \nu_j \ge 0 \text{ for }j = 1, \dots, m-1 \text{ and } \nu_1 + \dots + \nu_m = 0 \} \,. 
\]
We define a partial order on \( \ZZ^m\) by \( \lambda \ge \mu \) if and only if \( \lambda - \mu \in Q_+ \).

Given \(s\in\CC\) we define \((t-s)^\nu\) to be the diagonal matrix 
\[
\begin{bmatrix}
    (t-s)^{\nu_1} \\
    & (t-s)^{\nu_2} \\ 
    & & \ddots \\
    & & & (t-s)^{\nu_m}
\end{bmatrix} 
\]
which we can view in \(G(K)\) for any ring \( K \) containing \((t-s)^{-1}\) and \(\CC[t]\). For example, \((t-s)^\nu\in G(K)\) for \( K = \Ks\) or \(\CC(t)\).

We will also be interested in the affine space \( M_m\) of \(m\times m\) matrices.  Note that for any \( \CC\)-algebra \( R \), \(G(R)\) consists of those matrices \( g \in M_m(R) \) whose determinant is invertible in \( R\). Thus for example \( (t-s)^\mu\in M_m(\CC[t])\) for all effective \( \mu \in \ZZ^m\) but \((t-s)^\mu\in G(\CC[t])\) if and only if \( \mu = 0 \).

\subsection{Lattices}\label{ss:lat}
We will use the lattice model for the affine Grassmannian, so it is useful to recall the following definition. Let \( R \subset K\) be two \(\CC\)-algebras (usually, but not always, \(K\) will be a field). 
Consider \( K^m \) as a \(K\)-module. By restriction \( K^m\) can be viewed as an \(R\)-module.  An \new{\(R\)-lattice} in \(K^m\) is an \(R\)-submodule \( L \subset K^m\) which is a free \(R\)-module of rank \( m \) and satisfies \( L \otimes_R K = K^m \). Equivalently, \( L = \Sp_R(v_1, \dots, v_m)\) where \(v_1, \dots, v_m\) are free generators of \(K^m\). 

\( R^m \subset K^m \) is called the standard lattice. The group \(\GL_m(K) \) acts transitively on the set of \(R\)-lattices in \(K^m\), thus giving a bijection between this set and \(\GL_m(K)/\GL_m(R)\), since \( \GL_m(R) \) is the stabilizer of the standard lattice. 

We will be particularly interested in \(\CC[t]\)-lattices in \( \CC(t)^m\). Given such a lattice \( L \) and a point \( a \in \CC\), the \new{specialization} of \( L \) at \( a \) is the lattice in \(\cK_a^m\) defined as \( L(a) := L \otimes_{\CC[t]} \cO_a \). 
If \(L(a) = \cO_a^m\) then \(L\) is said to be \new{trivial at \(a\)}. 
For example, the lattice \((t-s)^{-1}\CC[t] \subset \CC(t)\) is trivial at any \(a\ne s\), since \( t-s \) is invertible in \( \cO_a\). 

\section{Affine Grassmannians}\label{s:affgrs}
We now define various versions of the affine Grassmannian which will play important roles in this paper. 
Each definition is made group-theoretically and then restated as a moduli space of vector bundles and as a moduli space of lattices.
We also sketch how to pass between descriptions.  

In these definitions, \( \Vtriv\) denotes the trivial rank \( m \) vector bundle. 

\begin{definition}\label{def:gr}
     The \new{ordinary affine Grassmannian} \(\Gr = G(\cK)/G(\cO)\).
\end{definition}    
It is the moduli space of vector bundles with trivializations,
\[
\Gr = 
    \left \{ 
        (V, \varphi) : \text{\(V\) is a rank \(m\) vector bundle on \( D_0 \), \(\varphi : V \xrightarrow{\sim} \Vtriv \) on \( D_0^\times \)} 
    \right \} \,. 
\]
It is also a moduli space of lattices, 
\[ 
\Gr = 
    \left\{ L \subset \cK^m : \text{ \(L\) is a \(\cO\)-lattice} \right\}\,.
\]
We obtain a lattice from a pair \( (V,\varphi) \) by setting \( L = \Gamma(D_0, V)\) which is embedded into \( \cK^m = \Gamma(D_0^\times, \Vtriv)\) using \( \varphi\).  
On the other hand, to get a lattice from the group-theoretic description \( G(\cK) / G(\cO) \) we set \( L = g \cO^m\) for \( g \in G(\cK)\).
\begin{definition}\label{def:grs}
    For any \( s \in \CC \), the \new{ordinary affine Grassmannian} \(\Gr_s = G(\cK_s)/G(\cO_s)\) \new{at} \( s \). 
\end{definition}  
As above, we have modular and lattice descriptions:
\begin{gather*}
\Gr_s = 
    \left\{ (
        V, \varphi) : \text{\(V\) is a rank \(m\) vector bundle on \( D_s \), \(\varphi : V \xrightarrow{\sim} \Vtriv \) on \( D^\times_s \)} 
    \right\}\,, \\
\Gr_s = 
    \left\{ 
        L \subset \cK_s^m : \text{ \(L\) is a \(\cO_s\)-lattice} 
    \right\} \,. 
\end{gather*}
\begin{definition}\label{def:grth}
    The \new{thick affine Grassmannian} \(\Grth = G(\Kinf)/G(\CC[t])\).
\end{definition}
Again, we have the following modular and lattice descriptions:
\begin{gather*}
\Grth = 
    \left\{ 
        (V, \varphi) : \text{\(V\) is a rank \(m\) vector bundle on \( \PP \), \(\varphi : V \xrightarrow{\sim} \Vtriv \) on \( D_\infty \) } 
    \right\}\,, \\
\Grth = 
    \left\{ 
        L : L \subset  \Kinf^m \text{ is a \(\CC[t]\)-lattice} 
    \right\}\,.
\end{gather*}
\begin{definition}\label{def:bdgr}
    The two point \new{Beilinson--Drinfeld Grassmannian} \[\pi : \Grbd\to \AA\] with one point fixed at 0, and the second point \(s\in\AA\) varying.
\end{definition}
It is described in modular terms by 
\[
\Grbd = 
    \left\{ 
        (V,\varphi,s) : V\text{ is a rank \(m\) vector bundle on }\PP, \varphi : V \xrightarrow{\sim} \Vtriv \text{ on } \PP \setminus \{0, s\}  
    \right\}\,. 
\]
The fibre of \(\Grbd \to \AA\) over \( s \in \AA \) will be denoted \( \Gr_{0,s} \) and is given by
\[  
    G(\CC[t, t^{-1}, {(t-s)}^{-1}]/G(\CC[t])\,. 
\]
We also have the lattice descriptions:
\begin{gather*}
\Grbd = 
    \left \{ 
        (L,s) : L \subset  \CC(t)^m\text{ is a \(\CC[t]\)-lattice trivial at any }a \ne 0, s 
    \right \} \,, \\
\Gr_{0,s} = 
    \left \{ 
        L : L \subset  {\CC[t,t^{-1},{(t-s)}^{-1}]}^m \text{ is a \(\CC[t]\)-lattice} 
    \right \} \,.
\end{gather*}
\begin{definition}\label{def:grplus}
The \new{positive part} of \(\Gr \), resp.\ \(\Grth\), is defined by 
\begin{gather*}
    \Gr^+ = \left(M_m(\cO) \cap G(\cK)\right) / G(\cO)\,,\,\, 
    \text{resp.\ } {\Grth}^+ = \left(M_m(\CC[t]) \cap G(\Kinf)\right) / G(\CC[t])\,. 
\end{gather*}
\end{definition}
In modular terms, \(\Gr^+\) (resp.\ \({\Grth}^+\)) is the set of those \( (V, \varphi)\) where \( \varphi : V \rightarrow \Vtriv \) extends to an inclusion of coherent sheaves over \( D_0 \) (resp.\ over \( \PP\)).

In lattice terms, \( \Gr^+\) (resp.\ \({\Grth}^+\)) contains those lattices \(L\) which are contained in \(\cO^m\) (resp.\ \(\CC[t]^m\)).

\subsection{Relations between different affine Grassmannians}\label{ss:relsbtwgrs}
These different versions of the affine Grassmannian are related as follows.  
\begin{proposition}\label{pr:bd-th}
    There is a map \(\Grbd \rightarrow \Grth \) defined in the following equivalent ways:
    \begin{enumerate}
        \item in modular terms as \[(V,\varphi,s)\mapsto (V, \varphi \big|_{D_\infty}) \, ;\]
        \item in the fibre over \( s \in \AA\) as the inclusion 
        \[
            G(\CC[t, t^{-1}, (t-s)^{-1}])/ G(\CC[t]) \rightarrow G(\Kinf)/G(\CC[t])\,;
        \]
        \item in terms of lattices as \[ (L,s)\mapsto L \,,\] using the inclusion \(\CC[t, t^{-1}, (t-s)^{-1}]^m \rightarrow \Kinf^m\) on ambient spaces. 
    \end{enumerate}
\end{proposition}

The following result is the factorization property of the BD Grassmannian (see \cite[Prop. 3.13]{zhu2016introduction}).
\begin{proposition}\label{pr:polyno-taylor}
    The fibres \( \Gr_{0,s}\) of \( \Grbd \rightarrow \AA\) can be described as 
    \begin{equation*}
    \Gr_{0,s} \cong 
        \begin{cases} 
            \Gr \times \Gr_s & s \ne 0 \\
            \Gr              & s = 0\,.
        \end{cases}
    \end{equation*}
    In the modular realization, this isomorphism is given by restricting the vector bundle and trivialization.
    
    Suppose that \( s \ne 0 \). In the lattice realization, this is given by forming the specializations 
    $$
        L \mapsto (L(0), L(s))\,.
    $$
    Note that if \( L = g \CC[t]^m\) for \( g \in G([t,t^{-1}, (t-s)^{-1}])\), then 
    $$
        (L(0), L(s)) = (g \cO^m, g\cO^m_s)\,. 
    $$
    In the \( s = 0 \) case, the isomorphism is described in the same way, except that we just need to form \( L(0)\).
\end{proposition}

For \( s \ne 0\), there is an isomorphism \( \Gr_s \cong \Gr \) coming from the isomorphism \( \Ks \cong \cK\). 
Combining with \Cref{pr:polyno-taylor}, we obtain an isomorphism
$$
    \theta_s : \Gr_{0,s} \rightarrow \Gr \times \Gr 
$$

Similarly, we write \( \theta_0\) for the isomorphism \( \Gr_{0,0} \xrightarrow{\sim} \Gr \).

\subsection{The fusion construction}\label{ss:fuscon}
The following construction will be very important in this paper.  

Let \( \Gr_{0, \AA^\times} \) denote the preimage of \(\Ax = \AA \setminus \{0\}\) under the map \(\Gr_{0,\AA} \rightarrow \AA \).  

The isomorphisms \( \theta_s \) defined above glue together to a projection map
\begin{equation}\label{eq:mvitau}
    \theta : \Gr_{0, \AA^\times} \rightarrow \Gr \times \Gr \, .
\end{equation}
(This \(\theta\) matches the map \( \tau \) of \cite{mirkovic2007geometric}, except that we are working over \( \{0 \} \times \AA^\times\) instead of the complement of the diagonal in \( \AA\times\AA \).)

Let \( X_1, X_2 \subset \Gr\) be two subschemes, and consider the subscheme 
$$ 
    X_1 \bdfamo X_2 := \theta^{-1}(X_1 \times X_2) \subset \Gr_{0, \AA^\times} 
$$
Given a point \( s \in \Ax\), we write \( X_1 \bdfamf X_2  \) for the fibre of \( X_1 \bdfamo X_2 \) over \(s\). The map \( \theta_s\) identifies \( X_1 \bdfamf X_2 \subset \Gr_{0,s}\) with \( X_1 \times X_2 \subset \Gr \times \Gr\).

We define \( X_1 \ast_{\AA} X_2 \) to be the scheme-theoretic closure of \(  X_1 \bdfamo X_2 \) in \( \Grbd \). By construction, this is a flat family over \( \AA\).

The fibre \( X_1 \geofu X_2\) of \( X_1 \ast_{\AA} X_2 \) over \(0\) is a subscheme of \( \Gr_{0,0} \), but regarded as a subscheme of \( \Gr \) using {the isomorphism} \(\Gr_{0,0} \cong \Gr \).  We call this subscheme the (geometric) \new{fusion} of \( X_1 \) and \( X_2\).

\subsection{Some subschemes of affine Grassmannians}\label{ss:subgrs}

Going forward, we fix a pair of arbitrary effective dominant coweights $\lambda', \lambda'' $ of sizes $N',N''$ and a pair of (not necessarily dominant) effective coweights $\mu',\mu''$ also of sizes $N',N''$ such that $\mu = \mu' + \mu''$ is dominant. 
Let $ \lambda = \lambda' + \lambda'' $ and $ N = N'+N''$.

Using $\lambda,\lambda',\lambda''$ and $\mu,\mu',\mu''$ we define the subschemes of $\Gr$, $\Grth$, and $\Grbd$ which will be considered, give some properties, and say how they are related.  
\begin{definition}\label{def:sphschub}
    The \new{spherical Schubert cell} $\Gr^\lambda = G(\cO) t^\lambda$ and its closure 
    $ \overline\Gr^\lambda = \bigcup_{\gamma \le \lambda} \Gr^{\gamma} $ 
    a \new{spherical Schubert variety}.  
\end{definition}
\begin{definition}\label{def:sphfus}
    The \new{family of two spherical Schubert varieties} $\overline{\Gr}^{\lambda', \lambda''}_{0, \AA} = \overline{\Gr}^{\lambda'} \bdfamc \overline{\Gr}^{\lambda''} $.
\end{definition}
By a theorem of Zhu, $ \overline{\Gr}^{\lambda'} \geofu \overline{\Gr}^{\lambda''} $ is reduced and equal to $ \overline{\Gr}^{\lambda}$ (\cite[Proposition 3.1.14]{zhu2016introduction}).
If $s\ne0$, then the fibre $\overline{\Gr}_{0,s}^{\lambda', \lambda''}$ contains the open locus $ \Gr_{0,s}^{\lambda', \lambda''} =  \Gr^{\lambda'} \bdfamf \Gr^{\lambda''} $ a $ G(\CC[t])$-orbit whose lattice description is given in \Cref{le:Grl1l2} below.
Because we consider effective dominant $\lambda',\lambda''$, the spherical Schubert varieties and their fusions lie in the positive part of the thick affine Grassmannian.
\begin{lemma}\label{le:sphfusispos}
    The map $ \Gr_{0, \AA} \rightarrow \Grth$ restricts to a map $ \overline{\Gr}^{\lambda', \lambda''}_{0, \AA} \rightarrow \Grth^+$.
\end{lemma}
\begin{proof}
    Let $ s \ne 0 $. Since $ \lambda', \lambda'' $ are effective, $ t^{\lambda'}, (t-s)^{\lambda''} \in \Grth^+ $.  Since $\Grth^+$ is closed and invariant under the action of $ G(\CC[t])$,  $ \overline{\Gr}_{0,s}^{\lambda', \lambda''} \subset \Grth^+$.
    
    For the $ s = 0$ fibre, a similar reasoning applies (or we can simply conclude by taking closure).
\end{proof}
\begin{definition}\label{def:klslice}
    The \new{Kazhdan--Lusztig slice} $\cW_\mu = G_1(\Oinf) t^\mu \subset \Grth $.
\end{definition}

In modular terms, $\cW_\mu$ corresponds to the locus of those $ (V, \varphi)$ such that $ V $ is isomorphic to the vector bundle $ \cO_{\PP}(\mu_1) \oplus \cdots \oplus \cO_{\PP}(\mu_m)$, and such that $ \varphi$ preserves the {Harder--Narasimhan filtration of $V$ at $ \infty$}. The lattice description of $ \cW_\mu \cap \Grth_+$ is given in Lemma \ref{le:Wmu} below. 
\begin{definition}\label{def:inftyorbit}
    The \new{semi-infinite orbit} $ S^\mu = N_-(\cK)t^\mu \subset \Gr $.
\end{definition}
\begin{definition}\label{def:inftyorbitfam}
    The \new{family of two semi-infinite orbits} $S^{\mu', \mu''}_{0,\AA} \subset \Grbd$ which we define to have fibres $ S^{\mu', \mu''}_{0,s} := \theta_s^{-1}( S^{\mu'}\times S^{\mu''}) $ for $ s \ne 0$ and fibre $ \theta_0^{-1}(S^{\mu})$ over $ s = 0 $.
\end{definition}
We can also describe the fibres as orbits
$$
    S^{\mu', \mu''}_{0,s} = 
    N_-(\CC[t, t^{-1}, (t-s)^{-1}])t^{\mu'} (t-s)^{\mu''} \subset G(\CC[t,t^{-1}, (t-s)^{-1}]) / G(\CC[t]) \,. 
$$
See \cite[Section~5.2]{baumann2020bases} for further details where this fibre is also described as the attracting locus for a $\Cx$ action.

In \cite[Prop~2.6]{kamnitzer2014yangians}, we observed that $ S^\mu \subset \cW_\mu$ (this uses that $\mu$ is dominant, which is the reason why we make this assumption).  More generally, we have the following result.

\begin{lemma}\label{le:inftyfusiskl} 
    Under the map $ \Grbd \rightarrow \Grth$, the image of $ S^{\mu', \mu''}_{0,\AA}$ lands in $ \cW_\mu$.
\end{lemma}
\begin{proof}
    Let $ L \in\Grth$ be a lattice in the image of $S^{\mu',\mu''}_{0,\AA}$. So we can write $ L = g t^{\mu'} (t-s)^{\mu''}\CC[t]^m$ for some $ g \in N_-(\CC[t, t^{-1}, (t-s)^{-1}]) $. 
    
    Let $ h =(t-s)^{\mu''} t^{-\mu''}  $.  Note that $ h \in T_1(\Oinf)$. 
    Moreover, $L = h (h^{-1} g h) t^{\mu}\CC[t]^m$. 
    Since $h^{-1} g h \in N_-(\cK_\infty)$ we can factor $ h^{-1} g h$ as $n_1 n_2$ for some $ n_1 \in {(N_-)}_1(\Oinf)$ and $n_2 \in {N_-}(\CC[t]) $. 
    
    As $ \mu $ is dominant, we see that $ t^{-\mu} n_2 t^\mu \in N_-(\CC[t]) $, and so $ L = h n_1 t^\mu \CC[t]^m$.  Since $ hn_1 \in G_1(\Oinf)$, the result follows. 
\end{proof}

\section{Matrices}\label{s:mats}
We now consider some subvarieties of the space of $ N\times N$ matrices, again using the coweights $ \lambda, \lambda', \lambda'', \mu, \mu', \mu''$ fixed in the previous section. 

\subsection{Adjoint orbits and their deformations}\label{ss:familiesofadjointorbits}
Recall that $\lambda=(\lambda_1,\dots,\lambda_m)$ is a partition of $ N$.
Given a point $ s \in \AA$ we write $ J_{s,\lambda}$ for the Jordan form matrix with eigenvalue $ s$ and Jordan blocks of sizes $ \lambda_1, \dots, \lambda_m$.
    
\begin{definition}\label{def:Olam}
The nilpotent (adjoint) orbit $ \OO^\lambda \subset M_N(\CC)$ of matrices conjugate to $ J_{0,\lambda}$. These matrices and the linear operators they represent will be said to have Jordan type $\lambda$.
\end{definition}  
\begin{definition}\label{def:Olamlam}
    For $ s \in \Ax$, the (adjoint) orbit $ \OO^{\lambda', \lambda''}_{0,s}$ of matrices conjugate to $ J_{0,\lambda'} \oplus J_{s,\lambda''}$.
    These matrices and the linear operators they represent will be said to have Jordan type $((0,\lambda'), (s,\lambda''))$.
\end{definition}  

We recall that these orbits have closures which are given by rank conditions.  
More precisely, we have
$$
    \overline{\OO}^\lambda = \{ A \in M_N(\CC) : \rk A^c \le N - \#\text{~boxes in first $c$ columns of }\lambda,  \text{ for $ c \in \NN$} \}
$$
and
\begin{equation}\label{eq:ranks}
\begin{split}
    \overline{\OO}^{\lambda', \lambda''}_{0,s} = \{ A \in M_N(\CC) : \rk A^c &\le N - 
    \#\text{~boxes in first $c$ columns of }\lambda',  \text{ for $ c \in \NN$} \\
    \rk (A-s)^c &\le N - \#\text{~boxes in first $c$ columns of }\lambda'',  \text{ for $ c \in \NN$} \}\,. 
\end{split}
\end{equation}

The following fact seems to be well-known but we could not find this exact statement in the literature.
\begin{proposition}\label{prop:adjoint}
    There exists a flat family $\overline{\OO}^{\lambda', \lambda''}_{0,\AA} \rightarrow \AA$ whose fibre over $s \in \AA$ is reduced and given by $\overline{\OO}^{\lambda', \lambda''}_{0,s}$ if $s \ne 0 $ and $\overline{\OO}^\lambda $ if $ s = 0$.
\end{proposition}
In order to prove this proposition, we will need to recall some results from Eisenbud--Saltman \cite{eisenbud1989rank}. 

Let $ r $ be a decreasing, non-negative function of $\NN$ with $r(0) = N$, called a \new{rank function}.  Let $k $ be maximal such that $ r(k) \ne 0 $.
Let $ W$ be a subspace of $ \AA^k$. Eisenbud and Saltman define and study the flat family $ X_{r,W} \rightarrow W$ whose fibres are reduced and defined by rank conditions. We will now describe these fibres. 

Let $ z \in \AA^k$ be a point. Some of the coordinates of $z$ may be equal, so we introduce the following data to keep track of these equalities. Let $ \ell$ denote the number of distinct coordinates of $ z $. There exist a sequence of integers $\underline{k} = (k_0,\dots,k_\ell)$ such that $ 0=k_0< k_1< \dots < k_\ell = k $, a point $ \underline s \in\AA^\ell$, such that $s_a\ne s_b$ for any $a\ne b$, and a permutation  $ p $ of $ \{1, \dots, k\}$ such that 
$$ 
    s_a = z_{p(k_{a-1}+1)} = z_{p(k_{a-1} + 2)} =  \cdots = z_{p(k_a)}
$$ 
for each $a = 1,\dots,\ell$. 
Given the choice of $ \underline s$ we require that $ p $ is of minimal length. 
Given $z$, the data $\underline{k}, \underline s, p$ is unique up to a permutation of $ \{1, \dots, \ell\}$.

Next for $ a = 1, \dots, \ell $ and $ c = 1, \dots, k_{a+1} - k_a$, we define
$$
    r_{z,a}(c):= N + \sum_{d = 1}^c r(p(k_a + d)) - r(p(k_a + d) - 1)  \, . 
$$
By \cite[Corollary 2.2]{eisenbud1989rank}, the fibre of the flat family $ X_{r,W} $ over $ z$ is given by
\begin{equation}\label{eq:ESfibre}
    X_{r,z} = \{ A \in M_N(\CC) :\rk (A-s_a)^c\le r_{z,a}(c) \text{ for all } c , a \text{ as above}\}.
\end{equation}
We will now apply these ideas to prove our \Cref{prop:adjoint}.
\begin{proof}
For $ c\in \NN$, let
$$
    r(c) = N - \#\text{~boxes in first $c$ columns of }\lambda\,.
$$
It is easy to see that this is a rank function, with $ k = \lambda_1$, the number of columns of $ \lambda$.

Since $\lambda = \lambda' + \lambda''$, there exists a permutation $ p $ of $ \{1, \dots, k\}$ (the columns of $ \lambda$) such that $ p(1), \dots, p(k_1) $ are the columns of $ \lambda'$ and $ p(k_1+1), \dots, p(k)$ are the columns of $ \lambda''$. Here $ k_1 = \lambda'_1$ the number of columns of $ \lambda'$.

For $ s \in \AA$, we define $ z(s) \in \AA^k$ by 
\[
\begin{cases}
    z(s)_{p(c)} = 0 &  c = 1, \dots, k_1 \\
    z(s)_{p(c)} = s &  c = k_1 + 1, \dots, k\,.
\end{cases}    
\]

For $ s \ne 0$, we see that the equalities in $ z(s) $ give rise to the data of \(\ell = 2\), \(\underline k = (0,\lambda_1')\), \(\underline s = ( 0, s)\) and the permutation $p$.
We also see that
\begin{equation}\label{eq:rcols}
    \begin{split}
            r_{z(s), 1}(c) &=  N - \#\text{~boxes in first $c$ columns of }\lambda' \\
            r_{z(s), 2}(c) &=  N - \#\text{~boxes in first $c$ columns of }\lambda'' \,. 
    \end{split}
\end{equation}
On the other hand, for $z(0)$, we get that $ \ell = 1$ and that 
\begin{equation}\label{eq:rcol2}
    r_{z(0),1}(c) = 
    r(c) = N - \#\text{~boxes in first $c$ columns of }\lambda\,. 
\end{equation}
Let $ W = \{ z(s) : s \in \CC \}$. Combining \Cref{eq:ranks,eq:ESfibre,eq:rcols,eq:rcol2}, we see that the Eisenbud--Saltman family $ X_{r, W} \rightarrow W $ gives our family $ \overline{\OO}^{\lambda', \lambda''}_{0, \AA}$.
\end{proof}
\subsection{The Mirkovi\'c--Vybornov slice}\label{ss:mvyslice} 
Recall that $\mu=(\mu_1,\dots,\mu_m)$ is a partition of $N$. The \new{Mirkovi\'c--Vybornov slice} $\TT_\mu$ is the affine space of $N\times N$ matrices of the form $A = J_{0,\mu} + X$ where $X$ is a $\mu\times\mu$ block matrix with possibly nonzero entries $A_{ij}^1,\dots,A_{ij}^{\min(\mu_i,\mu_j)}$ in the first $\min(\mu_i,\mu_j)$ columns of the last row of each $\mu_i\times\mu_j$ block. 

By example, if $\mu = (3,2)$ then $A\in\TT_\mu$ looks like 
\[
    \left[
        \begin{BMAT}(e){ccc;cc}{ccc;cc} 
        0 & 1 & 0 & 0 & 0\\
        0 & 0 & 1 & 0 & 0\\
        A_{11}^1 & A_{11}^2 & A_{11}^3 & A_{12}^1 & A_{12}^2\\
        0 & 0 & 0 & 0 & 1\\
        A_{21}^1 & A_{21}^2 & 0 & A_{22}^1 & A_{22}^2
        \end{BMAT}\right] 
\]
for some $A_{ij}^k\in\CC$. 

To $ A \in \TT_\mu$ we will associate the $m\times m$ matrix of polynomials $g(A)$ in $ M_m(\CC[t]) $ whose $(i,j)$th entry is defined as follows.
\begin{equation}\label{eq:mvyofa}
    g(A)_{ij} = 
        \begin{cases} 
            t^{\mu_i} - \sum_{k=1}^{\mu_i} A^k_{ji} t^{k-1} & i = j     \,, \\
            - \sum_{k=1}^{\mu_i} A^k_{ji} t^{k-1}           & i \ne j   \,. 
        \end{cases}
\end{equation}

Continuing with the $\mu=(3,2)$ example, 
\[
    g(A) = 
    \begin{bmatrix}
        t^{3} - A_{11}^3 t^2 - A_{11}^2 t - A_{11}^1    & -A_{21}^2t - A_{21}^1         \\
        -A_{12}^2 t - A_{12}^1                          & t^{2} - A_{22}^2 t - A_{22}^1
    \end{bmatrix}\,. 
\]

In other words, the $ \mu_i\times\mu_j$ block of $ A $ is used to produce a polynomial which is inserted in the $(j,i)$ entry of the $m\times m$ matrix $ g(A)$. 

In $\TT_\mu$ we will be interested in a certain family of block upper-triangular matrices.
\begin{definition} 
The \new{upper-triangular} \mvy slice $\UU^{\mu', \mu''}_{0,\AA}\rightarrow \AA $ is defined by
$$
    \UU^{\mu', \mu''}_{0,\AA} := \{ (A,s) \in \TT_\mu \times \AA : g(A)_{ii} = t^{\mu'_{i}} (t-s)^{\mu''_{i}}, g(A)_{ij} = 0 \text{ for $ j < i $ }\}\,. 
$$
\end{definition}
So a matrix in $\UU^{\mu', \mu''}_{0,\AA}$ is weakly block upper-triangular and its diagonal blocks are given by the companion matrices for the polynomials $\{t^{\mu'_{i}} (t-s)^{\mu''_{i}} : i=1,\dots,m\}$.

For example, elements of $ \UU_{0,s}^{(1,1,0),(2,1,1)}$ look like 
\[
    \left[
        \begin{BMAT}(e){ccc;cc;c}{ccc;cc;c} 
        0 & 1 & 0 & 0 & 0 & 0\\
        0 & 0 & 1 & 0 & 0 & 0\\
        0 & -s^2 & 2s & A_{12}^1 & A_{12}^2 & A_{13}^1 \\
        0 & 0 & 0 & 0 & 1 & 0\\
        0 & 0 & 0 & 0 & s & A_{23}^1\\
        0 & 0 & 0 & 0 & 0 & s
        \end{BMAT}
    \right] \,.   
\]
Note that the fibre $  \UU^{\mu', \mu''}_{0,0}$ is the same as the intersection $ \TT_\mu \cap \mathfrak n $, where $ \mathfrak n $ denotes the set of strictly upper-triangular matrices.  

\section{Rising Action}\label{s:rising}
Throughout this section, we continue our notation of the previous sections. So $ \lambda, \lambda', \lambda''$ denote dominant effective coweights, with $ \lambda' + \lambda'' = \lambda$.  Also $ \mu, \mu', \mu''$ are effective coweights with $ \mu = \mu' + \mu''$ and we assume that $ \mu$ is dominant.

\subsection{Linear Operators from lattices}\label{s:latticefacts}
\begin{lemma}\label{le:Grl1l2}
    Let $ L \in \Grth^+ $.  Let $ s \in \Ax$. The following are equivalent:
    \begin{enumerate}[label=(\roman*)]
        \item \label{it:schub-pos} $ L $ is in the image of the map $ \Gr^{\lambda', \lambda''}_{0,s} \rightarrow \Grth^+$. 
        \item \label{it:schub-jt} The linear operator $ t $ on $ \CC[t]^m/L$ has Jordan type $((0,\lambda'), (s,\lambda''))$.
        \item \label{it:schub-lat} $ L \in G(\CC[t]) t^{\lambda'} (t-s)^{\lambda''}$.
    \end{enumerate}
\end{lemma}
\begin{proof}
First we recall that for $ L \in \Gr^+$, $ L \in \Gr^{\lambda} = G(\cO)t^\lambda $ if and only if $ t |_{\cO^m/L} $ has Jordan type $ \lambda$. 

Now assume that $ (L, s) \in \Gr^{\lambda', \lambda''}_{0,s}$.  By definition, $ L(0) \in \Gr^{\lambda'}$ and $L(s) \in \Gr^{\lambda''} $.  This means that $t $ acting on $\cO^m/L(0)$ has Jordan type $ \lambda'$ and $ t$ acting on $\cO_s^m/L(s)$ has Jordan type $ \lambda''$.  For $ a = 0, s$,  we see that $$\CC[t]^m/L \otimes_{\CC[t]} \cO_a \cong \cO_a^m / L(a). $$
\Cref{le:linalg} shows that the map $\CC[t]^m/L\to \cO^m/L(0)$ induces an isomorphism between the $0$-generalized eigenspace of $ t$ and $ \cO^m / L(0)$.   The same thing holds for the $s$-generalized eigenspace of $t $ and $ \cO_s^m/L(s)$. This shows that \cref{it:schub-pos} implies \cref{it:schub-jt} and the logic can be reversed to see that \cref{it:schub-jt} implies \cref{it:schub-pos}. 

On the other hand, if $ L = g t^{\lambda'} (t-s)^{\lambda''} \CC[t]^m$ for some $ g \in G(\CC[t])$, then $ L(0) = g t^{\lambda'}\cO^m $ since $ (t-s)^{\lambda''} \in G(\cO)$. In this way, we see that \cref{it:schub-lat} implies \cref{it:schub-jt} and the logic can be reversed to get equivalence.
\end{proof}

\begin{lemma}\label{le:linalg} 
Let $ V $ be a $ \CC[t]$-module which is finite-dimensional as a complex vector space.  For any  $ a \in \AA$, the map
$$ 
    V  \rightarrow V \otimes_{\CC[t]} \cO_a
$$
restricts to an isomorphism between the generalized $ a $-eigenspace of $ t $ and $ V \otimes_{\CC[t]} \cO_a$
\end{lemma}
\begin{proof}
    For any $ b \in \AA$, let $ E_b $ denote the generalized $b$-eigenspace of $t$.  Then $ V = \oplus_{b \in \CC} E_b$.  Since $ t - b$ is invertible in $ \cO_a$ and $ t -b $ acts nilpotently on $ E_b$, we see that $ E_b \otimes_{\CC[t]} \cO_a = 0 $.
    
    So it suffices to show that $ E_a \rightarrow E_a \otimes_{\CC[t]} \cO_a$ is an isomorphism.  By the classification of modules over $ \CC[t]$, it suffices to check this when  $ E_a = \CC[t]/(t-a)^k$, where it is clearly true.
\end{proof}

\subsection{Matrices from lattices}\label{ss:mvyisos}

\begin{lemma}\label{le:Wmu}
Let $ L \in \Grth^+$.  The following are equivalent:
\begin{enumerate}[label=(\roman*)]
    \item \label{it:slice} $ L \in \cW_\mu$.
    \item \label{it:basis} $ L = \Sp_{\CC[t]}(v_1, \dots, v_m)$ for some $ v_i $ of the form $ v_i = t^{\mu_i} e_i + \sum_{j=1}^m p_{ij} e_j $ where $ p_{ij} \in \CC[t] $ has degree less than $ \min(\mu_i, \mu_j)$.
    \item \label{it:t-act} For all $ i $, 
    $$ t^{\mu_i} e_i \in \Sp_\CC(\{t^k e_j : 0 \le k < \min(\mu_i, \mu_j), 1 \le j \le m \}) + L. $$
\end{enumerate}
Moreover, for such $L $, $ \beta_\mu := \{ [t^k e_i] : 0 \le k < \mu_i, 1 \le i \le m\}$ forms a basis for $ \CC[t]^m/L$. 
\end{lemma}
\begin{proof}
    Let $ L \in \cW_\mu$.  Then $ L = \Sp_{\CC[t]}(v_1, \dots, v_m) $ for some $ v_i $ with $ v_i = t^{\mu_i} e_i + \sum_{j=1} q_{ij}t^{\mu_i} e_j $ and $ q_{ij} \in t^{-1} \Oinf$.  Since $ L \in \Grth^+ $, we see that $ v_i \in \CC[t]^m$ which means that $ p_{ij} := q_{ij}t^{\mu_i} $ lies in $ \CC[t]$.  By construction, the polynomial $ p_{ij}$ has degree less than $ \mu_i$.
    
    Fix $ i$ and suppose that for some $ j$, $ \mu_j < \mu_i$.  In this case, we can alter our basis to $ v'_i = v_i - q v_j$ for some polynomial $q \in \CC[t]$.  
    This gives us new polynomials $ p'_{ij} = p_{ij} - q (t^{\mu_j} + p_{jj}) $.  In this way, we can ensure that $ p_{ij} $ has degree less than $ \min(\mu_i, \mu_j)$.  
    Thus \cref{it:slice} implies \cref{it:basis}.
    
    Suppose that $ L = \Sp_{\CC[t]}(v_1, \dots, v_m)$ as in \cref{it:basis}.  Then
    $$t^{\mu_i} e_i - v_i \in \Sp_\CC(\{t^k e_j :  k < \min(\mu_i, \mu_j), 1 \le j \le m \})  \,. $$
    Hence \cref{it:basis} implies \cref{it:t-act}.  

    Finally, given \cref{it:t-act}, then we can see $ v_i := t^{\mu_i} e_i - \sum_{j=1}^m p_{ij} e_j \in L $ for some $ p_{ij} \in \CC[t]$ of degree less than $ \min(\mu_i,\mu_j) $.  It is easy to see that $ L = \Sp_\cO(v_1, \dots, v_m) $ and so $ L \in \cW_\mu$.  
    To show that $ \beta_\mu$ forms a basis for $ \CC[t]^m/L$, it suffices to show that for each $ i$, 
    $$ t^{\mu_i} e_i  \in \Sp_\CC(  t^k e_i : 0 \le k < \mu_i, 1 \le i \le m) + L\,.$$ This follows immediately from \cref{it:t-act}.
\end{proof}

Given $ A \in \TT_\mu$, recall the definition of $ g(A)$ given in \Cref{eq:mvyofa}.  Note that $g(A)t^{-\mu} \in G_1(\Oinf)$.
Since $ g(A) \in M_m(\CC[t]) \cap G_1(\Oinf)t^\mu$, we will regard $ g(A)$ as giving an element of $ \Grth^+ \cap \cW_\mu$.  (Alternatively, we can see that $ g(A) \CC[t]^m$ satisfies the condition of \cref{it:basis} from \Cref{le:Wmu}.)

The following result is called the Mirkovi\'c--Vybornov isomorphism \cite{mirkovic2007quiver}.  In its present form, it can be found in \cite[Theorem 3.2]{cautis2018categorical}, except that we have tweaked both maps with a matrix transpose $[\phantom{A}]^\tr$.

\begin{theorem}\label{th:TmuWmu}
The map $ \TT_\mu \rightarrow \Grth^+ \cap \cW_\mu $ given by $ A \mapsto g(A)\CC[t]^m $ is an isomorphism with inverse given by
$$ 
    L \mapsto [t|_{\CC[t]^m/L} ]^{\tr}_{\beta_\mu} \,. 
$$
\end{theorem}

\subsection{Upper-triangularity and the Mirkovi\'c--Vybornov isomorphism}\label{ss:upmvy}
For the next result, we will consider the ``intersection'' of $ \overline{\Gr}^{\lambda', \lambda''}_{0,\AA} $ with $\cW_\mu$. As $  \overline{\Gr}^{\lambda', \lambda''}_{0,\AA} $ is not a subscheme of $ \Grth$, by this intersection, we really mean the preimage of $ \cW_\mu$ under the composition
$$ 
    \overline{\Gr}^{\lambda', \lambda''}_{0,\AA}  \hookrightarrow \Grbd \rightarrow \Grth\,.
$$
This is not a very serious abuse of notation, since the map $ \Grbd \rightarrow \Grth $ is almost injective. In a similar way, we will write $ \overline{\OO}^{\lambda', \lambda''}_{0,\AA} \cap \TT_\mu$ using the ``almost injective'' map $ \overline{\OO}^{\lambda', \lambda''}_{0,\AA} \rightarrow M_N(\CC)$. 

The following refinement of the \mvy isomorphism is a special case of \cite[Theorem 5.3]{mirkovic2007quiver}. 
\begin{theorem}\label{th:OGrl}
There is an isomorphism
$$
    \overline{\OO}^{\lambda', \lambda''}_{0,\AA} \cap \TT_\mu \cong \overline\Gr^{\lambda', \lambda''}_{0,\AA} \cap \cW_\mu 
$$
given by $ (A,s) \mapsto (g(A)\CC[t]^m, s)$.
\end{theorem}
\begin{proof}
Since we already have the isomorphism from \Cref{th:TmuWmu}, it suffices to show that for any $ A \in \TT_\mu$, 
$$ 
    (A,s) \in \overline{\OO}^{\lambda', \lambda''}_{0,\AA} \cap \TT_\mu \text{ if and only if } (g(A)\CC[t]^m, s) \in \overline\Gr^{\lambda', \lambda''}_{0,\AA} \cap \cW_\mu \,. 
$$
This follows immediately from \Cref{le:Grl1l2}.
\end{proof}
\begin{theorem}\label{th:OTGrW}
The isomorphism from \Cref{th:OGrl} restricts to an isomorphism
$$ 
    \overline{\OO}^{\lambda', \lambda''}_{0,\AA} \cap \UU^{\mu', \mu''}_{0,\AA} \cong \overline{\Gr}^{\lambda', \lambda''}_{0,\AA} \cap S^{\mu', \mu''}_{0,\AA}\,.
$$
\end{theorem}
\begin{proof}
We could prove this by observing that both sides are the attracting locus of an appropriate $ \Cx$ action. However, we will give the following more algebraic proof.

Let $ A \in \TT_\mu$ and $ s \in \CC $. We must show that  $ (A,s) \in \UU^{\mu', \mu''}_{0,\AA} $ if and only if $ (g(A)\CC[t]^m,s) \in S^{\mu', \mu''}_{0,\AA} $. 

On the one hand, if $ (A,s) \in \UU^{\mu', \mu''}_{0,\AA} $, then $ g(A)$ is lower-triangular with diagonal $ t^{\mu'} (t-s)^{\mu''}$, and so $ g(A) \in N_-[t, t^{-1}, (t-s)^{-1}] t^{\mu'} (t-s)^{\mu''}$. 

On the other hand, if $ (g(A)\CC[t]^m, s) \in S^{\mu', \mu''}_{0,\AA}$, then we can write 
$$
    g t^\mu r= n t^{\mu'} (t-s)^{\mu''}
$$
for some $ r \in G(\CC[t]), n \in N_-(\CC[t, t^{-1}, (t-s)^{-1}]) $ and $ g = g(A)t^{-\mu}$.  Let $ h = (t-s)^{\mu''} t^{-\mu''}$ which lies in $ T_1(\Oinf) $. 
Note that $ h^{-1}n h \in N_-(\Kinf)$, so we can factor it as $ h^{-1} n h  = n_1 n_2 $, where $ n_1 \in N_{-,1}(\Oinf), n_2 \in N_-(\CC[t])$.  So then after doing a bit of algebra, we reach
$$
    t^\mu r (t^{-\mu} n_2^{-1} t^\mu) t^{-\mu} = g^{-1} h n_1.
$$
Since $ g, h, n_1 \in G_1(\Oinf)$, the right hand side $ g^{-1} h n_1 $ lies in $ G_1(\Oinf) $.  Since $ \mu $ is dominant, $ t^{-\mu} n_2 t^\mu \in N_-(\CC[t])$, and so the left hand side lies in $t^\mu G(\CC[t]) t^{-\mu}$.

Moreover, since $ \mu $ is dominant, we know that 
$$
    t^\mu G(\CC[t]) t^{-\mu} \cap G_1(\Oinf) = N_{-,1}(\Oinf)\,.
$$
Thus, we deduce that $ g^{-1} h n_1 \in N_{-,1}(\Oinf)$ and hence $ g(A) \in t^{\mu'} (t-s)^{\mu''}N_{-,1}(\Oinf) $. Since $ A \in \TT_\mu $, this implies that $ (A,s) \in \UU^{\mu', \mu''}_{0,\AA}$ as desired.
\end{proof}

Restricting to the zero fibre and applying $\theta_0$ together with \Cref{prop:adjoint}, we recover a result of the second author. 
\begin{corollary}\label{cor:mvy} (\cite[Corollary~5.2.2]{dthesis})
    The map $ A \mapsto g(A)\cO^m$ gives an isomorphism 
$$\overline{\OO}^\lambda \cap \TT_\mu\cap\n \cong\overline\Gr^\lambda\cap S^\mu_-\,.$$
\end{corollary}
\section{Climax}\label{s:climax}
\subsection{MV cycles and tableaux}\label{ss:mvcs}
Continuing now with the notation of the previous sections, we add the assumption that $\mu\le\lambda$ and forget the assumption that $\mu$ is dominant. 
\begin{definition}
    An irreducible component of $\overline{\Gr^\lambda \cap S^\mu}$ is called an \new{Mirkovi\'c--Vilonen cycle} in $\overline{\Gr}^\lambda$ of coweight $\mu$. 
\end{definition}

By \cite[Theorem~3.2(b)]{mirkovic2007geometric}, each MV cycle in $\overline\Gr^\lambda$ of coweight $\mu$ has dimension $ \rho(\lambda - \mu)$ where $ \rho = (m, m-1, \dots, 1)$ and we use the ordinary dot product on $\ZZ^m$. We denote by $\cZ(\lambda)$ the collection of all MV cycles in $\overline{\Gr}^\lambda$ of coweight $\mu\le\lambda$. 

In \cite{kamnitzer2010mirkovic}, the third author gave a combinatorial description of MV cycles for any reductive group, using Mirkovi\'c--Vilonen polytopes, or, equivalently, Lusztig data, which are sequences in $\NN^{r}$ that depend on a choice of reduced expression for the longest element of the Weyl group of $G$.  We will now explain how to use these results to index MV cycles for $ \GL_m$ using Young tableaux. (For any $ G$, we use $ r $ for the number of positive roots; in the case of $ \GL_m$, $ r = m(m-1)/2$.)

Denote by $[a,b]$ the interval $\{a,a+1,\dots,b-1,b\}\subset\NN$. We begin with the following definitions.
\begin{definition}
Let $ L \subset \cO^m$ be a point in $ \Gr_+$.  We define the relative dimension of $L $ by
$$
    \rdim L := \dim_\CC \cO^m/L\,. 
$$
If $ \gamma \subset [1,m]$, we define two lattices $ L^\gamma, L_\gamma $ in $ \cO^\gamma := \Sp_{\cO}(e_i : i \in \gamma)$ by
\begin{gather*}
    L^\gamma:= L \cap \cO^\gamma  \qquad L_\gamma := L / L^{\gamma^c} 
\end{gather*}
where $ \gamma^c := [1,m] \setminus \gamma$.
\end{definition}
We will also need the analogous definitions when $ \cO$ is replaced by $ \CC[t]$ and $L \in \Grth_+$. 
We record the following easy observations.
\begin{lemma}\label{le:rdimTrans} 
Let $ L \in \Gr_+$ or $\Grth_+$.
\begin{enumerate}
    \item For any $ \gamma \subset [m] $, we have 
    $$ 
        \rdim L_\gamma + \rdim L^{\gamma^c} = \rdim L \,.
    $$
    \item For any $ \gamma_1 \subset \gamma_2 \subset [m]$, we have 
    $$ 
        (L_{\gamma_2})_{\gamma_1} = L_{\gamma_1}\,.
    $$
\end{enumerate}
\end{lemma}
For any $ L \in \Gr_+ $ and $\gamma \subset [m]$, we define $ D_\gamma(L) := \rdim L^\gamma = \rdim L - \rdim L_{\gamma^c}$. (By \cite[Proposition 9.3]{kamnitzer2010mirkovic}, this coincides with the general definition of $ D_\gamma$ given in that paper; in \textit{loc. cit.} these functions are used to compute the hyperplanes of the associated MV polytopes.) 

From now on, we will write $ S^\mu_- = N_-(\cK) t^\mu$ (with a subscript), because we will also need the opposite semi-infinite cell $S^\mu_+ := N(\cK)t^\mu$.

\begin{lemma}\label{le:Srdim}
    Let $ L \in \Gr_+$.  The following are equivalent:
    \begin{enumerate}
        \item $ L \in S_+^\mu$ \label{Srdim1}
        \item $ \rdim L^{[1,i]} = \mu_1 + \dots + \mu_i$ for all $ i = 1, \dots, m $. \label{Srdim2} 
        \item $ \rdim L_{[i+1,m]} = \mu_{i+1} + \dots + \mu_m$ for $ i  = 1, \dots, m$. \label{Srdim3}
    \end{enumerate}
    Similarly, the following are equivalent:
        \begin{enumerate}[label=\arabic*'.]
        \item $ L \in S_-^\mu$
        \item $ \rdim L^{[i+1,m]} = \mu_{i+1} + \dots + \mu_m$ for all $ i = 1, \dots, m $.
        \item $ \rdim L_{[1,i]} = \mu_{1} + \dots + \mu_i$ for $ i  = 1, \dots, m$.
    \end{enumerate}
\end{lemma}
\begin{proof}
We prove the first statement as the second is similar.  Let $ L = n t^\mu \CC[t]^m$ for some $ n \in N_+(\cK)$ and let $ v_1, \dots, v_m $ denote the columns of the matrix $ nt^\mu$.  Then $ L = \Sp_{\cO}(v_1, \dots, v_m)$.  Fix $ i \le m$.  Then $ L^{[1,i]} = \Sp_{\cO}(w_1, \dots, w_i) $ where $ w_j $ denotes the first $ i $ entries of $ v_j $ (by upper-triangularity, the rest of the entries are 0, in any case).  

Assume that $L \subset \cO^m$.  Then $ L^{[1,i]} \subset \cO^i$ and $ \rdim L^{[1,i]} $ is the valuation of the determinant of the matrix whose columns are $ w_1, \dots, w_i$.  Since this matrix is upper-triangular with diagonal entries $ t^{\mu_1}, \dots, t^{\mu_i}$, this determinant is $ t^{\mu_1 + \dots + \mu_i} $ and hence $ \rdim L^{[1,i]} = \mu_1 + \dots + \mu_i$.   

Thus \cref{Srdim1} implies \cref{Srdim2}.  The converse follows from the fact that every $ L$ lies in some $ S_+^\nu$ and by the above reasoning, $ \nu$ is determined by the values of $ \dim L^{[1,i]}$.

The equivalence of \cref{Srdim2} and \cref{Srdim3} follows from \Cref{le:rdimTrans}.
\end{proof}


We now introduce some notation related to Young tableaux. Denote by $YT(\lambda)$ the set of (possibly semi-standard) Young tableaux of shape $\lambda$ and by $YT(\lambda)_\mu$ the subset of $YT(\lambda)$ of tableaux having weight $\mu$. Let $YT = \bigcup_\lambda YT(\lambda)$ and denote by $YT_+\subset YT$ the subset of those tableaux having dominant weight. 

Given $\tau\in YT(\lambda)_\mu$ and $ i \in \{1, \dots m\} $, denote by $\lambda(i)$ (resp.\ $\mu(i)$) the shape (resp.\ the weight) of $\tau(i)$, the tableau got from $\tau$ by discarding all boxes of weight exceeding $i$. (Note that $ \mu(i)$ only depends on $ \mu$, while $ \lambda(i)$ depends on the tableau $ \tau$.)  We will regard $ \lambda(i)$ (resp.\ $\mu(i)$) as an effective dominant coweight (resp.\ effective coweight) for $\GL_i$. 

The Lusztig datum $n_\bullet(\tau)$ of the tableau $\tau$ is a list of $r= m(m-1)/2$ non-negative integers defined from its \new{Gelfand--Tsetlin pattern} (see \cite[Sect.~4]{berenstein1988tensor}) $\gt(\tau) = (\lambda(i)_j)_{1\le j\le i\le m}$ by the formula 
$$
\begin{gathered}
    n_\bullet(\tau)_{(a,b)} = \lambda(b)_a - \lambda(b-1)_a = \text{ \# of boxes on row $a$ of $\tau $ of weight $b$}  \\
    (a,b)                   =  (1,2),\dots,(1,m),
                                (2,3),\dots,(2,m),
                                \dots,
                                (m-1,m)\,. 
\end{gathered}
$$
Note that $$ \lambda - \mu = \displaystyle{\sum_{1 \le a < b \le m}n_\bullet(\tau)_{(a,b)} \beta_{a,b}} $$ where $ \beta_{a,b} $ denotes the positive root of $ G$ with a $ 1 $ in the $a $ slot and a $-1$ in the $ b $ slot. The pattern $\gt(\tau)$ is recorded as a lower-triangular matrix (the array of shapes of subtableaux $\tau(i)$) and the datum $n_\bullet(\tau)$ is recorded as a sequence, unless noted otherwise. 

Below is an example with $\lambda = (4,2)$ and $\mu = (3,2,1)$. 
\[
    \tau = \young(1112,23) \qquad \gt(\tau) = 
    \begin{matrix}
        3 \\
        4 & 1 \\ 
        4 & 2 & 0         
    \end{matrix} \qquad n_\bullet(\tau) = (1,0,1) 
\]

We can associate to $\tau$ the locus 
\begin{equation*}
    \mathring Z(\tau) = 
        \{
            L\in S^\mu_- : L_{[1,i]} \in\Gr^{\lambda(i)}\text{ for } i = 1, \dots m 
        \}\,. 
\end{equation*}
\begin{remark} \label{rem:claxton}
Theorem 4.2 of \cite{claxton2015young} shows that the map $ YT(\lambda) \rightarrow \NN^r $ is the canonical crystal embedding $ B(\lambda) \rightarrow B(\infty)$.
\end{remark}

The following result is closely related to  Theorem 5.4.3 from \cite{dthesis}. It is also closely related to the description of MV cycles in terms of Kostant data obtained by Anderson--Kogan \cite{anderson2004mirkovic} (see \cite[Section 9]{kamnitzer2010mirkovic}).  
We remark that a different map from Young tableaux to MV cycles was obtained by Gaussent, Littelmann and Nguyen \cite[Theorem 2]{gaussent2013knuth}; we are not certain of the relation with our construction. 

\begin{proposition}\label{pr:newmvdes}
    $ \mathring Z(\tau)$ has a unique irreducible component of dimension $ \rho(\lambda - \mu)$.  Let $ Z(\tau)$ denote the closure of this component.  Then, 
    $Z(\tau)$ is the MV cycle whose Lusztig datum (with respect to the standard reduced word) is $n_\bullet(\tau)$. 
\end{proposition}
\begin{remark}\label{rem:irreducible}
We believe that $ \mathring Z(\tau)$ is irreducible, so that in fact $ Z(\tau) = \overline{ \mathring Z(\tau)}$.
\end{remark}

\begin{proof}
The proof of this proposition follows the same strategy as in \cite{dthesis}.
First, we consider 
\begin{equation*}
    \mathring Z(\tau)_1 := \{ L \in S^\mu_- \cap \Gr_+ : L_{[1,i]} \in S_+^{\lambda(i)} \text{ for } i = 1, \dots m\} \,. 
\end{equation*}
Using Lemma \ref{le:rdimTrans}(2) and Lemma \ref{le:Srdim}, we see that
$$
    \mathring Z(\tau)_1 = \{ L \in \Gr_+ : \rdim L_{[a,b]} = \lambda(b)_a + \dots + \lambda(b)_b \text{ for all $ 1 \le a \le b \le m $} \}
$$
where note that $ \lambda(b)_a + \dots + \lambda(b)_b$ is the number of boxes on rows $ a, \dots, b$ of weight $ 1, \dots, b $.

From  the proof of \cite[Proposition 9.6]{kamnitzer2010mirkovic}, we see that $ \mathring Z(\tau)_1$ is equal to $ t^\mu A(n_\bullet(\tau))$, where $ A(n_\bullet)$ is defined in  \cite[Sect.\ 4.3]{kamnitzer2010mirkovic}. In particular, it is irreducible and of dimension $ \rho(\lambda- \mu)$. Moreover, its closure is the MV cycle whose Lusztig datum is $ n_\bullet(\tau) $. (The reader is warned that in \cite{kamnitzer2010mirkovic} the third author worked with $  G(\cO) \backslash G(\cK)$, which we identify with $ \Gr $ using the map $ G(\cO) g \mapsto g^{-1} G(\cO)$.) 
Now, by results of \cite{kamnitzer2008hives} (see especially the proof of Theorem 1.4), we note that $\mathring Z(\tau)_1\cap \Gr^\lambda$ is dense in $ \mathring Z(\tau)_1$.  (This can also be deduced by combining \Cref{rem:claxton} with \cite[Proposition 6]{anderson2003polytope}.)

Fix $ i \in \{1, \dots, m\}$. If $ L \in S^\mu_-$, then $ L_{[1,i]} \in S^{\mu(i)}_-$ so we get a map 
$$ 
    f_i : \mathring Z(\tau)_1 \rightarrow S^{\mu(i)}_- \cap S_+^{\lambda(i)} \,. 
$$
From the definition of $ \mathring Z(\tau)_1 $, we see that $ f_i(\mathring Z(\tau)_1 )= \mathring Z(\tau(i))_1$.  
From above, $  f_i(\mathring Z(\tau(i))_1) \cap \Gr^{\lambda(i)}$ is a dense constructible subset of $\mathring Z(\tau(i))_1$.  
Since $ \mathring Z(\tau)_1$ is irreducible, this implies that $ f_i^{-1}(S^{\mu(i)}_- \cap S_+^{\lambda(i)} \cap \Gr^{\lambda(i)}) $ is a dense constructible subset of $ \mathring Z(\tau)_1 $.  

Working with all $ i $ at once, we conclude that
$$ 
    \mathring Z(\tau)_2 := \bigcap_{i=1}^m f_i^{-1}(S^{\mu(i)}_- \cap S_+^{\lambda(i)} \cap \Gr^{\lambda(i)}) 
$$
is a dense constructible subset of $\mathring Z(\tau)_1$. Thus $\overline{\mathring Z(\tau)_2} = \overline{\mathring Z(\tau)_1}$.

On the other hand, by definition 
$$
    \mathring Z(\tau)_2 \subset \mathring Z(\tau).
$$
Since $ \dim \mathring Z(\tau)_2 = \rho(\lambda - \mu)$, we see that $ \mathring Z(\tau) $ must have at least one component of the maximal dimension.  To see that it cannot have any other components of maximal dimension, we note that
$$ \bigsqcup_{\tau \in YT(\lambda)_\mu} \mathring Z(\tau) \subset \Gr^\lambda \cap S_-^\mu$$
and the number of irreducible components of the right hand side equals $|YT(\lambda)_\mu|$, thus on the left hand side, each $ \mathring Z(\tau)$ can only have one irreducible component of dimension $ \rho(\lambda - \mu)$.
\end{proof}
\subsection{Fusion of MV cycles via fusion of generalized orbital varieties}

Given $A \in M_N(\CC)$ we denote by $A\big|_{\CC^p}$ the restriction of $A$ to the subspace spanned by the first $p$ standard basis vectors of $\CC^N$.  If $A\big|_{\CC^p}(\CC^p)\subset\CC^p$ then we identify it with the $p\times p$ upper-left submatrix of $A$. 

Let $\tau\in YT(\lambda)_\mu$ with $\mu$ dominant. We define
\[
\mathring X(\tau) = 
    \{
        A \in  \TT_\mu\cap\n : A\big|_{\CC^{|\mu(i)|}} \in \OO^{\lambda(i)} \text{ for each } i = 1,\dots,m
    \}\,.
\]

\begin{lemma}\label{lem:XtZt}
    Under the Mirkovi\'c--Vybornov isomorphism $\mathring X(\tau)$ is mapped isomorphically onto $\mathring Z(\tau)$. 
\end{lemma}
\begin{proof}
Fix $A\in\mathring X(\tau)$ and $i\in\{1,\dots,m\}$. Let $\cO^i\subset\cO^m$ denote the submodule generated by the first $i$ standard basis vectors. Let $l = g(A\big|_{\CC^{\lvert\mu(i)\rvert}}) \cO^i$ and $L = g(A) \cO^m$. By definition of the map $A \mapsto g(A)$, since $ A $ is upper-triangular we have that $l = L_{[1,i]}$. 
Moreover, since elements of $\TT_\mu\cap\n$ are upper triangular $A\big|_{\CC^{\lvert\mu(i)\rvert}}\in\TT_{\mu(i)}\cap\n_i$ where $\n_i$ denotes the subalgebra of upper-triangular matrices in $M_{\lvert\mu(i)\rvert}(\CC)$. 
By definition of $\mathring X(\tau)$ this principal submatrix has Jordan type $\lambda(i)$. It follows that $l = L_{[1,i]} \in \Gr^{\lambda(i)} \cap S^{\mu(i)}$ for each $i=1,\dots,m$ so that $L\in\mathring Z(\tau)$. 

Conversely, let $ L \in \mathring Z(\tau) $.  By \Cref{cor:mvy}, $ L = g(A)\cO^m$ for some $ A \in \TT_\mu \cap \n $.  Running the above argument in reverse, we see that $ A \in \mathring X(\tau)$. 
\end{proof}

By \cite[Prop.\ 4.5.4]{dthesis} (or \Cref{pr:newmvdes} above), $ \mathring X(\tau) $ has a unique irreducible component of dimension $ \rho(\lambda - \mu)$.  We write $ X(\tau)$ for the closure of this component.  It is an irreducible component of $ \overline{\OO}^\lambda \cap \TT_\mu \cap \n$ and will be called a \new{generalized orbital variety of type $\lambda$}. 
The collection of all possible generalized orbital varieties of type $\lambda$ will be denoted $\cX(\lambda)$. 
\begin{theorem}\label{pr:generalized orbital varietiesasirrecs} (\cite[Theorem~4.8.2]{dthesis})
    $\{X(\tau) : \tau \in YT(\lambda)_\mu \}$ is a complete set of irreducible components of $\overline{\OO}^\lambda\cap\TT_\mu\cap\n$. 
\end{theorem}
Our next goal is to describe the fusion of two MV cycles using a ``fusion'' of generalized orbital varieties, which we will now define. 
Let $\lambda,\lambda',\lambda''$ and $\mu, \mu',\mu''$ be as in previous sections, once again assuming that $\mu$ is dominant.

Given a point $s\in\Ax$ and a pair of tableaux $\tau'\in YT(\lambda')_{\mu'}$ and $\tau''\in YT(\lambda'')_{\mu''}$ we define
\begin{equation*}
    \mathring X(\tau',\tau'')_{0,s} = \left\{A\in\UU_{0,s}^{\mu',\mu''} : A \big|_{\CC^{|\mu(i)|}} 
    \text{ has Jordan type } ((\lambda'(i), 0), (\lambda''(i), s)) \text{ for } i=1,\dots,m\right\}
\end{equation*}
and
\begin{equation*}
    \mathring X(\tau',\tau'')_{0,\AA^\times} = 
    \left\{
        (A,s)\in M_N \times \AA^\times: 
        A \in \mathring X(\tau',\tau'')_{0,s} 
    \right\}\,.
\end{equation*}
Note that elements $A\in\mathring X(\tau',\tau'')$ do not in general correspond to pairs $(A',A'')\in\mathring X(\tau')\times X(\tau'')$ because if $\mu'$ (resp.\ $\mu''$) is not dominant then $X(\tau')$ (resp.\ $X(\tau'')$) is not well-defined. 
\begin{proposition}\label{pr:XttZtt}
    The image of $\mathring X(\tau',\tau'')_{0,\AA^\times}$ under the isomorphism of \Cref{th:OTGrW} is $\mathring Z(\tau')\ast_{\AA^\times} \mathring Z(\tau'')$. 
\end{proposition}

\begin{proof}
Fix $A\in \mathring X(\tau',\tau'')_{0,s} $ and $i \in \{1,\dots, m\}$. 
Let $\CC[t]^i\subset\CC[t]^m = \CC^m\otimes_\CC\CC[t]$ denote the $\CC[t]$-submodule generated by the first $i$ standard basis vectors.

The lattices $l=g(A\big|_{\CC^{|\mu(i)|}})\CC[t]^i$ and $L=g(A)\CC[t]^m$ are related by the equation $ l = L_{[1,i]}$ where recall $L_{[1,i]} = (L + \CC[t]^{[i+1,m]}) / \CC[t]^{[i+1,m]}\subset \CC[t]^m/\CC[t]^{[i+1,m]} = \CC[t]^i$. 
    
By definition of $\UU^{\mu',\mu''}_{0,\AA}$ we have $A\big|_{\CC^{|\mu(i)|}}\in\UU^{\mu'(i),\mu''(i)}_{0,s}$ and by definition of $\mathring X(\tau',\tau'')_{0,s}$, the Jordan type of $A\big|_{\CC^{|\mu(i)|}}$ is $((\lambda'(i),0),(\lambda''(i),s))$. So, $l \in\Gr_{0,s}^{\lambda'(i),\lambda''(i)}$ by \Cref{th:OTGrW}. 
On the other hand, since $L\in S^{\mu',\mu''}_{0,s}$, $l\in S^{\mu'(i),\mu''(i)}_{0,s}$. 

Therefore, as $i$ varies, we see that the pair $(L(0),L(s))\in S^{\mu'}\times S^{\mu''}$ is such that 
$$
    (L_{[1,i]}(0),L_{[1,i]}(s))  = (L(0)_{[1,i]},L(s)_{[1,i]}) \in \Gr^{\lambda'(i)}\times\Gr^{\lambda''(i)}
$$ 
and so $(L(0),L(s))\in \mathring Z(\tau')\times \mathring Z(\tau'')$. 

Conversely, given $L\in\mathring Z(\tau')\ast_s \mathring Z(\tau'')$, we can write $ L = g(A)\CC[t]^m$ by  \Cref{th:OGrl}. Then reversing the above logic, we conclude that $ A \in \mathring X(\tau',\tau'')_{0,s}$.
\end{proof}

Now, in analogy with the fusion of \Cref{ss:fuscon} we define $X(\tau',\tau'')_{0,\AA}$ to be the Zariski closure of the unique top-dimensional component of $\mathring X(\tau',\tau'')_{0,\AA^\times}$ in $M_N \times\AA$.  (As noted in \cref{rem:irreducible}, we expect that all these varieties $ \mathring X(\tau) \cong \mathring Z(\tau) $ are irreducible, which would mean that taking this ``top-dimensional component'' irrelevant.)  Then we define the fusion of generalized orbital varieties to be the scheme-theoretic intersection 
$$ X(\tau', \tau'')_{0,0}  = X(\tau',\tau'')_{0,\AA} \cap M_N \times \{0\} \text{ in $M_N \times\AA$} $$
By \Cref{prop:adjoint}, it is contained in $\overline\OO^\lambda\cap\TT_\mu\cap\n$. 

We quickly recall the definition of intersection multiplicity that we will need following \cite[Example~2.6.5]{fulton2016intersection}. We consider a fibre square of $\CC$-schemes
$$
\begin{tikzcd}
   W\ar[r]\ar[d]&V\ar[d]\\D\ar[r]&Y
\end{tikzcd}
$$
with $D $ an effective Cartier divisor and $V$ an
irreducible variety of dimension $k$. We assume that $V$ is not contained
in the support of $D$. Let $Z$ be an irreducible component of $W$;
it is a subvariety of $V$ of codimension $1$. The \textbf{multiplicity}
of $Z$ in the intersection $D\cap V$ is defined to be the length of the
module $\mathcal O_{V,Z}/(f)$ over the local ring $\mathcal O_{V,Z}$
of $V$ along $Z$, where $f$ is a local equation of $D|_V$ on an affine
open subset of $V$ which meets $Z$. Following \cite[chap.~7]{fulton2016intersection},
this multiplicity is denoted by $i(Z,D\cdot V)$.  

For our purposes, the zero fibre in $\Grbd$ is an effective divisor $D$.   
For $ Z', Z'' $ subvarieties of $ \Gr $, $V = Z'\ast_\AA Z''\subset\Grbd$ is a variety not contained in the support of $D$.  Then we choose an irreducible component $Z$ of $W := D \cap V = Z'\ast_0 Z''$.  Thus, we may consider the intersection multiplicity $ i(Z, Z'\ast_0 Z'') := i(Z, D \cdot V)$.  
We will be particularly interested in the case when $ Z, Z', Z''$ are MV cycles.

In a similar way, we may consider the intersection multiplicities of generalized orbital varieties in $X(\tau', \tau'')_{0,0} = M_N \times \{0\} \cap X(\tau', \tau'')_{0, \AA}$.
\begin{corollary}\label{cor:intmul}
    $i(X(\tau), X(\tau',\tau'')_{0,0}) = i( Z(\tau), Z(\tau')\ast_0 Z(\tau''))\,.$
\end{corollary}

\begin{proof}
By the isomorphism of \Cref{lem:XtZt}, $X(\tau) $ is a dense constructible subset of $ Z(\tau) $.  On the other hand, \Cref{pr:XttZtt} gives us an identification of $ \mathring X(\tau',\tau'')_{0, \Ax}$ and $ \mathring Z(\tau') \ast_{\Ax} Z(\tau'')$; the schemes $ X(\tau', \tau'')_{0,\AA} $ and $ Z(\tau') \ast_\AA Z(\tau'') $ are constructed from these by taking the top-dimensional irreducible component and then closure.  Thus, we conclude that $ X(\tau', \tau'')_{0,\AA} $ is a dense constructible subset of $ Z(\tau') \ast_\AA Z(\tau'') $.  From the definition of intersection multiplicity, it is clear that the multiplicity can be computed locally and so the result follows.
\end{proof}
\subsection{Multiplying MV basis elements in \texorpdfstring{$\CC[N]$}{C[N]}}
\label{ss:CN}
In this section we will need to recall that representations of the Langlands dual group $G^\vee$ of an arbitrary complex reductive group $G$ can be constructed from the affine Grassmannian of $G$, and studied in $\CC[N^\vee]$, the ring of functions on the maximal unipotent subgroup $N^\vee$ of $G^\vee$. 

We use the convenient fact that $\GL_m^\vee = \GL_m$ (so also $N^\vee = N$) and ignore the distinction between weights and coweights. 

Let $N\subset G$ be the unipotent subgroup of upper-triangular matrices with 1s on the diagonal. In \cite{baumann2019mirkovic} the third author along with Baumann and Knutson show that the MV cycles yield a basis of the coordinate ring $\CC[N]$ called the MV basis. Note that with respect to the action of the maximal torus $T$ of $G$ by conjugation on $N$, $\CC[N]$ acquires a homogeneous grading by $Q_+$. 

Given weights $\mu\le\lambda$ with $\lambda$ dominant, denote by $V(\lambda)$ the irreducible representation of $G$ of highest weight $\lambda$ and by $V(\lambda)_\mu$ its $\mu$-weight space. 
There is an injective map $ \Psi_\lambda: V(\lambda) \to \CC[N]$ realizing $V(\lambda)$ as a subspace of $\CC[N]$ and sending $ V(\lambda)_\mu$ to $\CC[N]_{\mu - \lambda}$. 

The geometric Satake correspondence says that $V(\lambda)_\mu = H_{2\rho(\lambda-\mu)}(\overline{\Gr^\lambda\cap S^\mu_-})$. In particular, the set $\{[Z] : Z\in\cZ(\lambda)\}$ gives a basis of $V(\lambda)$ with $[Z]$ denoting the class of $Z$ in the appropriate homology group. 

By a theorem of Anderson, MV cycles in $\overline{\Gr^\lambda\cap S^{\nu + \lambda}_-}$ consist of those irreducible components of $ \overline{S^\lambda_+ \cap S^{\lambda+\nu}_-} $ that are contained in $\overline\Gr^\lambda$ \cite[Proposition~3]{anderson2003polytope}. The \new{stable} MV cycles of coweight $\nu\in Q_+$ are defined to be irreducible components of $\overline{S^0_+ \cap S^\nu_-}$.  Multiplication by $ t^\lambda $ defines an isomorphism $\overline{S^0_+ \cap S^\nu_-} \cong \overline{S^\lambda_+ \cap S^{\lambda+\nu}_-} $.  We denote the set of all stable MV cycles $\cZ(\infty)$. 
By a theorem of the third author, given a choice a reduced word for $ w_0 $, there is a bijection $ \cZ(\infty) \rightarrow \NN^{r} $ called the Lusztig datum.  For our purposes we will always use the standard reduced word.

One can string together the maps $\Psi_\lambda$ and the geometric Satake isomoprhisms to show that the stable MV cycles yield a basis of $\CC[N]$.  

\begin{theorem}
(\cite[Proposition~6.1]{baumann2019mirkovic})
    For each $Z\in\cZ(\infty)$ there is a unique element $b_Z\in\CC[N]$ such that for any dominant weight $\lambda$, if $t^\lambda Z\in\cZ(\lambda)$ then $\Psi_\lambda([t^\lambda Z]) = b_Z$.
\end{theorem}
The set $ \{b_Z : Z \in \cZ(\infty) \}$ is called the \new{MV basis} of $ \CC[N]$.  The structure constants of multiplication in $\CC[N]$ with respect to the MV basis are given by intersection multiplicities. 
\begin{theorem}
(\cite[Theorem~7.11]{baumann2019mirkovic}) 
    Given $Z',Z''\in\cZ(\infty)$, 
\begin{equation}\label{eq:btimesb}
    b_{Z'}\arfu b_{Z''} = \sum_{Z\in\cZ(\infty)} i\left(Z, Z' \geofu Z''\right) b_Z 
\end{equation}
in $\CC[N]$. 
\end{theorem}
Our next goal is to show that these structure constants can be computed using generalized orbital varieties. Consider the following commutative diagram.
\[
\begin{tikzcd}
    YT_+ \ar[r,"\tau\mapsto X(\tau)"]\ar[d,hook] & \bigcup \mathcal X(\lambda) \ar[d,hook,"\text{\Cref{lem:XtZt}}"] & \\
    YT      \ar[r,"\tau\mapsto Z(\tau)"] \ar[d,"n_\bullet"] & \bigcup\cZ(\lambda) \ar[r] \ar[d, "t^{-\lambda}"] & \bigoplus V(\lambda) \ar[d, "\Psi_\lambda"] \\ 
    \NN^{r} \ar[r] \ar[u,dashed,bend left,"\sigma"]         & \cZ(\infty) \ar[r]                                & \CC[N] 
\end{tikzcd}
\]

The sequence $\mu^0 = (\mu^0_1,\dots,\mu^0_{m-1})$ keeping track of the number of boxes of weight $i$ in row $i$ of a given tableau is called its \new{padding}. 
Two tableaux have equal Lusztig data if and only if they are related by padding, which is to say that we can change the padding of one (removing or adding boxes) to recreate the other. Moreover, a tableau is completely determined by its Lusztig datum and padding.

For example 
$$
    \bar \tau = \young(2,4) \quad \text{ and} \quad \tau = \young(112,24,3)
$$ 
have equal Lusztig data, since we can increase the padding of $\bar\tau$ (by adding to it the padding of $\tau$) to get $\tau$, or forget the padding of $\tau$ to get $\bar\tau$.

We call a tableau \new{stable} if its padding cannot be decreased. The following lemma shows that stable tableaux are in bijection with Lusztig data. 

\begin{lemma}\label{lem:almostmintab}
Let $n_\bullet = (n_{(a,b)})$ be a Lusztig datum and let $\mu^0 = (\mu^0_i)$ be a padding. If $\lambda = (\sum_b n_{(a,b)} + \mu^0_a)_a $ then the smallest $\mu^0$ such that $\lambda$ is dominant effective (and $\mu$ is effective) is:
$$
\begin{gathered}
        \mu^0_{m}, \mu^0_{m-1} = 0 \\
        \mu^0_i = \max\{0, \mu_{i+1}^0 + \sum_a n_{(a,i+1)} - \sum_a n_{(a,i)}\} \quad i = 1,\dots,m-2
\end{gathered}
$$
This choice of $\mu^0$ defines a section $\sigma:\NN^{r} \to YT$.
\end{lemma}
\begin{proof}[Proof sketch.]
To produce a tableau with Lusztig datum $n_\bullet$ of smallest possible shape and weight:
\begin{itemize}
    \item we can always take the number of $m$'s in row $m$, and the number of $m-1$'s in row $m-1$ to be zero; 
    \item we can take the number of $m-2$'s in row $m-2$ to be zero, unless $n_\bullet$ tells us that there are more boxes in row $m-1$ than there are in row $m-2$, and then we take $\mu^0_{m-2}$ to offset the difference; 
\end{itemize}
and iterate the last step up to $\mu^0_1$. 
\end{proof}
\begin{corollary}\label{cor:goodreps}
Given two stable MV cycles $Z',Z''$ with Lusztig data $n_\bullet',n_\bullet''$ resp.\ there exists a pair of tableaux $(\tau',\tau'')\in YT(\lambda')_{\mu'}\times YT(\lambda'')_{\mu''}$ for some $\lambda',\lambda''$ dominant effective and some $\mu',\mu''$ effective, such that $\mu = \mu' + \mu''\le \lambda = \lambda' + \lambda''$ are partitions and
\begin{equation}\label{eq:multingeneralized orbital varieties}
    b_{Z'} \arfu b_{Z''} = \sum_{\tau\in YT(\lambda)} i (X(\tau), X(\tau',\tau'')_{0,0})b_{t^{-\lambda}Z(\tau)}
\end{equation}
\end{corollary}
\begin{proof}
Choose $\tau' = \sigma(n_\bullet')$, $\tau'' = \sigma(n_\bullet'')$ and suppose $\tau'\in YT(\lambda')_{\mu'}$, $\tau'' \in YT(\lambda'')_{\mu''}$. 
Note that $Z'= t^{-\lambda'}Z(\tau'), Z'' = t^{-\lambda''} Z(\tau'')$ so the intersection multiplicity of a stable MV cycle $Z$ in $Z'\geofu Z''$ can be computed as the intersection multiplicity of an MV cycle $Z(\tau)$ in $Z(\tau')\geofu Z(\tau'')$ for some $\tau\in YT(\lambda)_\mu$, because these two situations isomorphic by translation by $ t^\lambda$. 

In case $\mu$ is not dominant, increase the padding of the larger of the two tableaux by $(\max(0,\mu_2-\mu_1),\max(0,\mu_3 - \mu_2),\dots,\max(0,\mu_m-\mu_{m-1}))$. 

Now, by \Cref{cor:intmul}, the intersection multiplicity of $Z(\tau)$ in $Z(\tau')\geofu Z(\tau'')$ can in turn be computed as the intersection multiplicity of the generalized orbital variety $X(\tau)$ in $X(\tau',\tau'')_{0,0}$.  (We have ensured $ \mu $ is dominant, so that $X(\tau)$ will be defined.) 
\end{proof}
\section{Examples}\label{s:examples}

In this section, we will compute some examples of multiplication of MV basis elements using the method provided by \Cref{cor:goodreps}.

We use the section defined by \Cref{lem:almostmintab} to abbreviate MV basis elements to tableaux and view $\CC[N]$ as an algebra in these, rewriting \Cref{eq:btimesb} as an equation in tableaux.
The coefficients are found using generalized orbital varieties.
Suppose we have two tableaux $\tau'$ and $\tau''$, with respective weights $\lambda'$, $\mu'$ and $\lambda''$, $\mu''$, and we wish to form their (geometric) fusion product. Once we have applied any necessary padding (as in the proof of \Cref{cor:goodreps}) so that $\mu = \mu' + \mu''$ is dominant, we take a generic matrix $A \in X(\tau', \tau'')_{0,s}$.

The requirement that $ A_i := A\big|_{\CC^{|\mu(i)|}} \in \OO^{\lambda'(i), \lambda''(i)}_{0,s}$ for each $i$ imposes rank conditions on $A$ in the form of vanishing and non-vanishing minors. We form the ideal $I\subset \CC[A_{ij}^k, s^\pm] = \CC[\UU_{0, \Ax}^{\mu', \mu''}]$ generated by these relations. The ideal of the fusion $ X(\tau', \tau'')_{0,0}$ is
$$
    J = (I \cap \CC[A_{ij}^k,s]) + (s)\,.
$$ 
Taking the primary decomposition (to account for possible multiplicities) of $J$ gives us ideals $J_1,\dots,J_n$, each corresponding to a generalized orbital variety occurring in the fusion. We use \Cref{pr:generalized orbital varietiesasirrecs} to find a tableau for each $J_i$ and then forget all unnecessary padding to get the corresponding stable MV cycle. 
Each $J_i$ also tells us the multiplicity of each MV cycle in the fusion product, where we have multiplicity 1 if and only if $J_i = \sqrt {J_i}$.

\subsection{Type \texorpdfstring{$A_2$}{A2}}\label{ss:a2}
\begin{example}
$\young(22)\arfu\young(11,33)= \young(33) + \young(22,33) + 2 \, \young(23,3)= \young(2)\,^2 + \young(2,3)\,^2 + 2 \, \young(2,3)\cdot\young(3)\,$:
%
The requirement that 
\[
A = \left[\begin{BMAT}(e){cc;cc;cc}{cc;cc;cc}
    0 & 1 & & & & \\
    -s^2 & 2s & A_{12}^1 & A_{12}^2 & A_{13}^1 & A_{13}^2 \\
     & & 0 & 1 & & \\
     & & & 0 & A_{23}^1 & A_{23}^2 \\
     & & & & 0 & 1 \\
     & & & & -s^2 & 2s
\end{BMAT}
\right]\text{ is contained in } \young(22)\bdfamo \young(11,33)
\]
imposes the following nontrivial relations on $A$. 
{\small
$$
\begin{aligned}
    & 2 A_{12}^2 A_{23}^2 s + 3 A_{13}^2 s^2 + A_{12}^2 A_{23}^1 + A_{12}^1 A_{23}^2 + 2 A_{13}^1 s, A_{13}^2 s^3 - A_{12}^2 A_{23}^1 s - A_{12}^1 A_{23}^2 s - 2 A_{12}^1 A_{23}^1, \\ 
     & A_{12}^2 A_{13}^2 A_{23}^1 s^2 + A_{12}^1 A_{13}^2 A_{23}^2 s^2 - (A_{12}^2 A_{23}^1)^2 + 2 A_{12}^1 A_{12}^2 A_{23}^1 A_{23}^2 - (A_{12}^1 A_{23}^2)^2 + 6 A_{12}^1 A_{13}^2 A_{23}^1 s + 4 A_{12}^1 A_{13}^1 A_{23}^1, \\
     & A_{12}^1 A_{13}^2 (A_{23}^2)^2 s^2 - (A_{12}^2 A_{23}^1)^2 A_{23}^2 + 2 A_{12}^1 A_{12}^2 A_{23}^1 (A_{23}^2)^2 - (A_{12}^1)^2 (A_{23}^2)^3 - 2 A_{12}^2 A_{13}^2 (A_{23}^1)^2 s \\
     & \quad + 4 A_{12}^1 A_{13}^2 A_{23}^1 A_{23}^2 s - A_{13}^1 A_{13}^2 A_{23}^1 s^2 - 3 A_{12}^1 A_{13}^2 (A_{23}^1)^2 + 4 A_{12}^1 A_{13}^1 A_{23}^1 A_{23}^2, \\
     & (A_{12}^2)^3 (A_{23}^1)^2 A_{23}^2 - 2 A_{12}^1 A_{23}^1 (A_{12}^2 A_{23}^2)^2 + A_{12}^2 (A_{12}^1)^2 (A_{23}^2)^3 + 2 A_{13}^2 (A_{12}^2 A_{23}^1)^2 s + 2 A_{13}^2 (A_{12}^1 A_{23}^2)^2 s \\
     & \quad + 3 A_{12}^1 (A_{13}^2)^2 A_{23}^1 s^2 + A_{13}^1 (A_{12}^2 A_{23}^1)^2 + 4 A_{12}^1 A_{12}^2 A_{13}^2 (A_{23}^1)^2 - 6 A_{12}^1 A_{12}^2 A_{13}^1 A_{23}^1 A_{23}^2 \\
     &\quad + 4 (A_{12}^1)^2 A_{13}^2 A_{23}^1 A_{23}^2  - 4 A_{12}^1 A_{13}^1 A_{13}^2 A_{23}^1 s - 4 A_{12}^1 (A_{13}^1)^2 A_{23}^1 + A_{13}^1 (A_{12}^1 A_{23}^2)^2
\end{aligned}
$$}
At $s=0$ the ideal generated by the above relations decomposes as a union of the following generalized orbital varieties, where the last component occurs with multiplicity 2. 
\begin{table}[H]
  \centering
  \newcolumntype{Q}{>{$}l<{$}>{$\small}l<{$}}
  \begin{tabular}{Q} 
    \text{Ideal of }X(\tau) & \tau \\
    \midrule 
    ( A_{12}^1, A_{12}^2,s ) & \young(1133,22)  \BS \\ 
    ( A_{23}^1, A_{23}^2,s ) & \young(1122,33) \BS\TS \\
    \sqrt{( (A_{23}^1)^2, A_{12}^2A_{23}^1 + A_{12}^1A_{23}^2, A_{12}^1A_{23}^1, (A_{12}^1)^2,s )} & \young(1123,23) \TS
    \end{tabular}
\end{table}
\end{example}

\subsection{Type \texorpdfstring{$A_3$}{A3}}\label{ss:GL4 examples}
There is a cluster structure on $\CC[N] \cong \CC[x_{ij} : 1 \leq i < j \leq 4]$ consisting of 12 cluster variables. The initial cluster coming from the standard reduced word is
\[
\begin{array}{ccc}
    x_1 = x_{12} & x_2 = x_{13} & x_3 = x_{12}x_{23}-x_{13} \\
    x_4 = x_{12}x_{23}x_{34} - x_{12}x_{24} - x_{13}x_{34} + x_{14} & x_5 = x_{13}x_{24}-x_{14}x_{23} & x_6 = x_{14} 
\end{array}
\]
We summarize our findings in the table below, listing each cluster variable along with its stable tableau, the coordinate ring of its generalized orbital variety, and the geometry of its MV cycle.  
\newcolumntype{H}{>{\setbox0=\hbox\bgroup$}c<{$\egroup}@{}}
\begin{table}[H]
  \centering
  \newcolumntype{Q}{>{$}l<{$}H>{$\small}l<{$}>{$}l<{$}>{$}l<{$}>{$}l<{$}}
  \begin{tabular}{Q} 
    n_\bullet & \tau = \text{Dom. } \overline \tau & \tau  & \text{Coordinate ring of }X(\tau) & Z &  \text{Cluster variable} \\
    \midrule 
(1,0,0,0,0,0) &\young(12) & \young(2) &  \CC[A_{12}^1] & \PP^1 &  x_1 \rule{0pt}{3ex}\\
(0,0,0,1,0,0) & \young(11,23) & \young(1,3) &  \frac{\CC[A_{12}^1,A_{13}^1,A_{23}^1]}{( A_{12}^1,A_{13}^1 )} & \PP^1 &  \frac{x_2 + x_3}{x_1} \TS\BS\\
(0,0,0,0,0,1) & \young(11,22,34) & \young(1,2,4) &  \frac{\CC[A_{12}^1,A_{12}^2,A_{13}^1,A_{14}^1,A_{23}^1,A_{24}^1,A_{34}^1]}{( A_{12}^1,A_{12}^2,A_{13}^1,A_{14}^1,A_{23}^1,A_{24}^1 )} & \PP^1 &  \frac{x_2x_4 + x_3x_6 + x_1x_5}{x_2x_3} \TS\BS\\
(1,0,0,1,0,0) & \young(13,2) & \young(2,3) & \frac{\CC[A_{12}^1,A_{13}^1,A_{23}^1]}{( A_{23}^1 )}  & \PP^2 &  x_2 \TS\BS\\
(0,1,0,0,0,0) & \young(12,3) & \young(3) &  \frac{\CC[A_{12}^1,A_{13}^1,A_{23}^1]}{( A_{12}^1 )} & \PP^2 &  x_3 \TS\BS\\
(0,0,0,1,0,1) & \young(11,24,3) & \young(1,3,4) & \frac{\CC[A_{12}^1,A_{13}^1,A_{14}^1,A_{23}^1,A_{24}^1,A_{34}^1]}{( A_{12}^1,A_{13}^1,A_{14}^1,A_{34}^1 )}  & \PP^2 &  \frac{x_1x_5 + x_6(x_2+x_3)}{x_1x_2} \TS\BS\\
(0,0,0,0,1,0) & \young(11,23,4) & \young(1,4) &  \frac{\CC[A_{12}^1,A_{13}^1,A_{14}^1,A_{23}^1,A_{24}^1,A_{34}^1]}{( A_{12}^1,A_{13}^1,A_{23}^1,A_{34}^1 )} & \PP^2 &  \frac{x_1x_5 + x_4(x_2+x_3)}{x_1x_3}\TS\BS\\
(0,1,0,0,0,1) & \young(112,24,3) & \young(13,2,4) &  \frac{\CC[A_{12}^1,A_{13}^1,A_{14}^1,A_{23}^1,A_{24}^1,A_{34}^1]}{( A_{12}^1,A_{34}^1,A_{13}^1A_{24}^1-A_{23}^1A_{14}^1 )} & S(2,4) &  \frac{x_3x_6+x_1x_5}{x_2} \TS\BS\\
(1,0,0,0,1,0) & \young(13,2,4) & \young(2,4) &  \frac{\CC[A_{12}^1,A_{12}^2,A_{13}^1,A_{14}^1,A_{23}^1,A_{24}^1,A_{34}^1]}{( A_{12}^1,A_{13}^1,A_{23}^1,A_{24}^1 )} & S(2,4) &  \frac{x_2x_4 + x_1x_5}{x_3} \TS\BS\\
(0,0,1,0,0,0) & \young(12,3,4) & \young(4) &  \frac{\CC[A_{12}^1,A_{13}^1,A_{14}^1,A_{23}^1,A_{24}^1,A_{34}^1]}{( A_{12}^1,A_{13}^1,A_{23}^1 )} & \PP^3 &  x_4 \TS\BS\\
(0,1,0,0,1,0) & \young(13,24) & \young(3,4) &  \frac{\CC[A_{12}^1,A_{13}^1,A_{14}^1,A_{23}^1,A_{24}^1,A_{34}^1]}{( A_{12}^1,A_{34}^1 )} & Gr(2,4) &  x_5 \TS\BS\\
(1,0,0,1,0,1) & \young(14,2,3) & \young(2,3,4) &  \frac{\CC[A_{12}^1,A_{13}^1,A_{14}^1,A_{23}^1,A_{24}^1,A_{34}^1]}{( A_{23}^1,A_{24}^1,A_{34}^1 )} & \PP^3 & x_6 \TS\BS\\
  \end{tabular}
  \caption{The variety $S(2,4)$ denotes the non-smooth Schubert divisor in $Gr(2,4)$ the Grassmannian of planes in $\CC^4$}
\end{table}

In what follows we check that the exchange relations satisfied by the above cluster variables in $\CC[N]$ are corroborated by the fusion of the corresponding MV cycles.
\begin{example}
$\young(2) \cdot \young(1,3) = \young(3) + \young(2,3)$\,: 
The requirement that 
\[
A = \left[\begin{BMAT}(e){c;c;c}{c;c;c}
    s & A_{12}^1 & A_{13}^1 \\
     & 0 & A_{23}^1 \\
     & & s
\end{BMAT}
\right] \text{ is contained in } \young(2) \bdfamo \young(1,3) 
\]
imposes the relation $A_{12}^1A_{23}^1+sA_{13}^1$ on $A$ and at $s = 0$ the ideal generated by this relation decomposes as a union of the following generalized orbital varieties.
\begin{table}[H]
  \centering
  \newcolumntype{Q}{>{$}l<{$}>{$}l<{$}}
  \begin{tabular}{Q} 
    \text{Ideal of } X(\tau) & \tau \\ 
    \midrule 
    (A_{12}^1,s) & \young(13,2) \BS \\
    ( A_{23}^1,s) & \young(12,3) \TS
    \end{tabular}
\end{table}
\end{example}

\begin{example}
$\young(1,3) \cdot \young(1,2,4) = \young(1,4) + \young(1,3,4)\,$: The requirement that
\[
A = \left[\begin{BMAT}(e){cc;c;c;c}{cc;c;c;c}
    0 & 1 & & & \\
     & s & A_{12}^1 & A_{13}^1 & A_{14}^1 \\
     & & s & A_{23}^1 & A_{24}^1 \\
     & & & 0 & A_{34}^1 \\
     & & & & s
\end{BMAT}
\right] \text{ is contained in } \young(1,3) \bdfamo \young(1,2,4)
\]
results in the following relations on submatrices:
\begin{table}[H]
  \centering
  \newcolumntype{Q}{>{$}l<{$}>{$}l<{$}}
  \begin{tabular}{Q} 
    \text{Submatrix} & \text{Relations} \\
    \midrule 
    A_2 & A_{12}^1 \\
    A_3 & A_{13}^1 \\
    A_4 & A_{14}^1, A_{23}^1A_{34}^1 + sA_{24}^1
    \end{tabular}
\end{table}
\noindent At $s = 0$ the ideal generated by the above relations decomposes as a union of the following generalized orbital varieties.
\begin{table}[H]
  \centering
  \newcolumntype{Q}{>{$}l<{$}>{$\small }l<{$}}
  \begin{tabular}{Q} 
    \text{Ideal of } X(\tau) & \tau \\ 
    \midrule 
    (A_{12}^1,A_{13}^1,A_{23}^1,A_{14}^1,s) & \young(11,24,3) \BS \\
    (A_{12}^1,A_{13}^1,A_{14}^1,A_{34}^1,s) & \young(11,23,4) \TS 
    \end{tabular}
\end{table}
\end{example}

\begin{example}
$\young(2) \cdot \young(1,3,4) = \young(2,3,4) + \young(13,2,4)\,$: The requirement that
\[
A = \left[\begin{BMAT}(e){c;c;c;c}{c;c;c;c}
    s & A_{12}^1 & A_{13}^1 & A_{14}^1 \\
     & 0 & A_{23}^1 & A_{24}^1 \\
     & & s & A_{34}^1 \\
     & & & s
\end{BMAT}
\right] \text{ is contained in } \young(2) \bdfamo \young(1,3,4)
\]
results in the following relations on submatrices:
\begin{table}[H]
  \centering
  \newcolumntype{Q}{>{$}l<{$}>{$}l<{$}}
  \begin{tabular}{Q} 
    \text{Submatrix} & \text{Relations} \\
    \midrule 
    A_3 & A_{12}^1A_{23}^1 + sA_{13}^1 \\
    A_4 & A_{34}^1, A_{12}^1A_{24}^1 + sA_{14}^1, A_{23}^1A_{14}^1 - A_{13}^1A_{24}^1
    \end{tabular}
\end{table}
\noindent At $s = 0$ the ideal generated by the above relations decomposes as a union of the following generalized orbital varieties.
\begin{table}[H]
  \centering
  \newcolumntype{Q}{>{$}l<{$}>{$\small}l<{$}}
  \begin{tabular}{Q} 
    \text{Ideal of } X(\tau) & \tau \\ 
    \midrule 
    (A_{23}^1,A_{24}^1,A_{34}^1,s) & \young(12,3,4) \BS \\
    (A_{12}^1,A_{34}^1,A_{23}^1A_{14}^1 - A_{13}^1A_{24}^1,s) & \young(13,2,4) \TS
    \end{tabular}
\end{table}
\end{example}

\begin{example}
$\young(2) \cdot \young(11,24,3) = \young(4) + \young(2,4)\,$: The requirement that
\[
A = \left[\begin{BMAT}(e){cc;cc;c;c}{cc;cc;c;c}
    0 & 1 & & & & \\
    -s^2 & 2s & A_{12}^1 & A_{12}^2 & A_{13}^1 & A_{14}^1 \\
     & & 0 & 1 & & \\
     & & & s & A_{23}^1 & A_{24}^1 \\
     & & & & s & A_{34}^1 \\
     & & & & & s
\end{BMAT}
\right] \text{ is contained in } \young(2) \bdfamo \young(11,24,3)
\]
results in the following relations on submatrices:
\begin{table}[H]
  \centering
  \newcolumntype{Q}{>{$}l<{$}>{$}l<{$}}
  \begin{tabular}{Q} 
    \text{Submatrix} & \text{Relations} \\
    \midrule 
    A_2 & A_{12}^1 + sA_{12}^2 \\
    A_3 & A_{13}^1, A_{23}^1 \\
    A_4 & A_{12}^2A_{24}^1 + sA_{14}^1, A_{12}^1A_{24}^1 - s^2A_{14}^1 
    \end{tabular}
\end{table}
\noindent At $s = 0$ the ideal generated by the above relations decomposes as a union of the following generalized orbital varieties.
\begin{table}[H]
  \centering
  \newcolumntype{Q}{>{$}l<{$}>{$\small}l<{$}}
  \begin{tabular}{Q} 
    \text{Ideal of } X(\tau) & \tau \\ 
    \midrule 
    (A_{12}^1,A_{12}^2,A_{13}^1,A_{23}^1,s) & \young(114,22,3) \BS \\
    (A_{12}^1,A_{13}^1,A_{23}^1,A_{24}^1,s) & \young(112,24,3) \TS
    \end{tabular}
\end{table}
\end{example}

\begin{example}

$\young(11,2,4) \cdot \young(2,3) = \young(2,4) + \young(2,3,4)\,$: The requirement that
\[
A = \left[\begin{BMAT}(e){cc;cc;c;c}{cc;cc;c;c}
    0 & 1 & & & & \\
     & 0 & A_{12}^1 & A_{12}^2 & A_{13}^1 & A_{14}^1 \\
     & & 0 & 1 & & \\
     & & & s & A_{23}^1 & A_{24}^1 \\
     & & & & s & A_{34}^1 \\
     & & & & & 0
\end{BMAT}
\right] \text{ is contained in } \young(11,2,4) \bdfamo \young(2,3)
\]
results in the following relations on submatrices:
\begin{table}[H]
  \centering
  \newcolumntype{Q}{>{$}l<{$}>{$}l<{$}}
  \begin{tabular}{Q} 
    \text{Submatrix} & \text{Relations} \\
    \midrule 
    A_2 & A_{12}^1 \\
    A_3 & A_{23}^1 \\
    A_4 & A_{24}^1, A_{13}^1A_{34}^1 - sA_{14}^1 
    \end{tabular}
\end{table}
\noindent At $s = 0$ the ideal generated by the above relations decomposes as a union of the following generalized orbital varieties.
\begin{table}[H]
  \centering
  \newcolumntype{Q}{>{$}l<{$}>{$\small}l<{$}}
  \begin{tabular}{Q} 
    \text{Ideal of } X(\tau) & \tau \\ 
    \midrule 
    (A_{12}^1,A_{13}^1,A_{23}^1,A_{24}^1,s) & \young(112,24,3) \BS \\
    (A_{12}^1,A_{23}^1,A_{24}^1,A_{34}^1,s) & \young(112,23,4) \TS
    \end{tabular}
\end{table}
\end{example}

\begin{example}
$\young(1,2,4) \cdot \young(3) = \young(4) + \young(13,2,4)\,$: The requirement that
\[
A = \left[\begin{BMAT}(e){c;c;c;c}{c;c;c;c}
    0 & A_{12}^1 & A_{13}^1 & A_{14}^1 \\
     & 0 & A_{23}^1 & A_{24}^1 \\
     & & s & A_{34}^1 \\
     & & & 0
\end{BMAT}
\right] \text{ is contained in } \young(1,2,4) \bdfamo \young(3)
\]
results in the following relations on submatrices:
\begin{table}[H]
  \centering
  \newcolumntype{Q}{>{$}l<{$}>{$}l<{$}}
  \begin{tabular}{Q} 
    \text{Submatrix} & \text{Relations} \\
    \midrule 
    A_2 & A_{12}^1 \\
    A_4 & A_{23}^1A_{34}^1 - sA_{24}^1, A_{13}^1A_{34}^1 - sA_{14}^1, A_{13}^1A_{24}^1 - A_{23}^1A_{14}^1 
    \end{tabular}
\end{table}
\noindent At $s = 0$ the ideal generated by the above relations decomposes as a union of the following generalized orbital varieties.
\begin{table}[H]
  \centering
  \newcolumntype{Q}{>{$}l<{$}>{$\small}l<{$}}
  \begin{tabular}{Q} 
    \text{Ideal of } X(\tau) & \tau \\ 
    \midrule 
    (A_{12}^1,A_{13}^1,A_{23}^1,s) & \young(14,2,3) \BS \\
    (A_{12}^1,A_{34}^1,A_{13}^1A_{24}^1-A_{23}^1A_{14}^1,s) & \young(13,2,4) \TS
    \end{tabular}
\end{table}
\end{example}

\begin{example}

$\young(11,24) \cdot \young(3) = \young(1,3)\cdot \young(4) + \young(3,4)\,$: The requirement that 
\[
A = \left[\begin{BMAT}(e){cc;c;c;c}{cc;c;c;c}
    0 & 1 & & & \\
     & 0 & A_{12}^1 & A_{13}^1 & A_{14}^1 \\
     & & 0 & A_{23}^1 & A_{24}^1 \\
     & & & s & A_{34}^1 \\
     & & & & 0
\end{BMAT}
\right] \text{ is contained in } \young(11,24) \bdfamo \young(3)
\]
results in the following relations on submatrices:
\begin{table}[H]
  \centering
  \newcolumntype{Q}{>{$}l<{$}>{$}l<{$}}
  \begin{tabular}{Q} 
    \text{Submatrix} & \text{Relations} \\
    \midrule 
    A_2 & A_{12}^1 \\
    A_4 & A_{13}^1A_{34}^1 - sA_{14}^1
    \end{tabular}
\end{table}
\noindent At $s = 0$ the ideal generated by the above relations decomposes as a union of the following generalized orbital varieties.
\begin{table}[H]
  \centering
  \newcolumntype{Q}{>{$}l<{$}>{$\small}l<{$}}
  \begin{tabular}{Q} 
    \text{Ideal of } X(\tau) & \tau \\ 
    \midrule 
    (A_{12}^1,A_{13}^1,s) & \young(114,23) \BS \\
    (A_{12}^1,A_{34}^1,s) & \young(113,24) \TS
    \end{tabular}
\end{table}
\end{example}

\begin{example}

$\young(11,23,4) \cdot \young(2,3) = \young(3,4) + \young(2,3,4)\cdot \young(1,3) \,$: The requirement that 
\[
A = \left[\begin{BMAT}(e){cc;cc;cc;c}{cc;cc;cc;c}
    0 & 1 & & & & & \\
     & 0 & A_{12}^1 & A_{12}^2 & A_{13}^1 & A_{13}^2 & A_{14}^1 \\
     & & 0 & 1 & & & \\
     & & & s & A_{23}^1 & A_{23}^2 & A_{24}^1 \\
     & & & & 0 & 1 & \\
     & & & & & s & A_{34}^1 \\
     & & & & & & 0
\end{BMAT}
\right] \text{ is contained in } \young(11,23,4) \bdfamo \young(2,3)
\]
results in the following relations on submatrices:
\begin{table}[H]
  \centering
  \newcolumntype{Q}{>{$}l<{$}>{$}l<{$}}
  \begin{tabular}{Q} 
    \text{Submatrix} & \text{Relations} \\
    \midrule 
    A_2 & A_{12}^1 \\
    A_3 & A_{12}^2A_{23}^1 - sA_{13}^1, A_{23}^1 + sA_{23}^2, A_{12}^2A_{23}^2+A_{13}^1 \\
    A_4 & A_{34}^1, A_{12}^2A_{24}^1 - sA_{14}^1, A_{14}^1A_{23}^1 - A_{13}^1A_{24}^1
    \end{tabular}
\end{table}
\noindent At $s = 0$ the ideal generated by the above relations decomposes as a union of the following generalized orbital varieties.
\begin{table}[H]
  \centering
  \newcolumntype{Q}{>{$}l<{$}>{$\small}l<{$}}
  \begin{tabular}{Q} 
    \text{Ideal of } X(\tau) & \tau \\ 
    \midrule 
    (A_{12}^1,A_{12}^2,A_{13}^1,A_{23}^1,A_{34}^1,s) & \young(113,224,3) \BS \\
    (A_{12}^1,A_{23}^1,A_{24}^1,A_{34}^1,A_{12}^2A_{23}^2+A_{13}^1,s) & \young(112,233,4) \TS
    \end{tabular}
\end{table}
\end{example}

\begin{example}
$\young(113,22,4) \cdot \young(1,3) = \young(3,4) + \young(1,3,4)\arfu \young(3)\,$: The requirement that
\[
A = \left[\begin{BMAT}(e){ccc;cc;cc;c}{ccc;cc;cc;c}
    0 & 1 & & & & & & \\
     & 0 & 1 & & & & & \\
     & & s & A_{12}^1 & A_{12}^2 & A_{13}^1 & A_{13}^2 & A_{14}^1 \\
     & & & 0 & 1 & & & \\
     & & & & 0 & A_{23}^1 & A_{23}^2 & A_{24}^1 \\
     & & & & & 0 & 1 & \\
     & & & & & & s & A_{34}^1 \\
     & & & & & & & 0
\end{BMAT}
\right] \text{ is contained in } \young(113,22,4) \bdfamo \young(1,3)
\]
results in the following relations on submatrices:
\begin{table}[H]
  \centering
  \newcolumntype{Q}{>{$}l<{$}>{$}l<{$}}
  \begin{tabular}{Q} 
    \text{Submatrix} & \text{Relations} \\
    \midrule 
    A_2 & A_{12}^1, A_{12}^2 \\
    A_3 & A_{13}^1 + sA_{13}^2 \\
    A_4 & A_{34}^1, A_{14}^1 A_{23}^1 - A_{13}^1 A_{24}^1
    \end{tabular}
\end{table}
\noindent At $s = 0$ the ideal generated by the above relations decomposes as a union of the following generalized orbital varieties.
\begin{table}[H]
  \centering
  \newcolumntype{Q}{>{$}l<{$}>{$\small}l<{$}}
  \begin{tabular}{Q} 
    \text{Ideal of } X(\tau) & \tau \\ 
    \midrule 
    (A_{12}^1,A_{12}^2,A_{13}^1,A_{23}^1,A_{34}^1,s) & \young(1113,224,3) \BS \\
    (A_{12}^1,A_{12}^2,A_{13}^1,A_{14}^1,A_{34}^1,s) & \young(1113,223,4) \TS
    \end{tabular}
\end{table}
\end{example}

\begin{example}
$\young(1,3) \cdot \young(2,4) = \young(3,4) + \young(1,4)\arfu\young(2,3)\,$: The requirement that 
\[
A = \left[\begin{BMAT}(e){c;c;c;c}{c;c;c;c}
    0 & A_{12}^1 & A_{13}^1 & A_{14}^1 \\
     & s & A_{23}^1 & A_{24}^1 \\
     & & 0 & A_{34}^1 \\
     & & & s
\end{BMAT}
\right] \text{ is contained in } \young(1,3) \bdfamo \young(2,4)
\]
results in the following relations on submatrices:
\begin{table}[H]
  \centering
  \newcolumntype{Q}{>{$}l<{$}>{$}l<{$}}
  \begin{tabular}{Q} 
    \text{Submatrix} & \text{Relations} \\
    \midrule 
    A_3 & A_{12}^1A_{23}^1 - sA_{13}^1 \\
    A_4 & A_{23}^1A_{34}^1 + sA_{24}^1, A_{12}^1A_{24}^1 + A_{13}^1A_{34}^1
    \end{tabular}
\end{table}
\noindent At $s = 0$ the ideal generated by the above relations decomposes as a union of the following generalized orbital varieties.
\begin{table}[H]
  \centering
  \newcolumntype{Q}{>{$}l<{$}>{$\small}l<{$}}
  \begin{tabular}{Q} 
    \text{Ideal of } X(\tau) & \tau \\ 
    \midrule 
    (A_{12}^1,A_{34}^1,s) & \young(13,24) \BS \\
    (A_{23}^1,A_{12}^1A_{24}^1 + A_{13}^1A_{34}^1,s) & \young(12,34) \TS
    \end{tabular}
\end{table}
\end{example}

\begin{example}

$\young(113,2,4) \cdot \young(2,3) = \young(3,4)\arfu \young(2) + \young(2,3,4)\arfu\young(3)\,$: The requirement that 
\[
A = \left[\begin{BMAT}(e){cc;cc;cc;c}{cc;cc;cc;c}
    0 & 1 & & & & & \\
     & 0 & A_{12}^1 & A_{12}^2 & A_{13}^1 & A_{13}^2 & A_{14}^1 \\
     & & 0 & 1 & & & \\
     & & & s & A_{23}^1 & A_{23}^2 & A_{24}^1 \\
     & & & & 0 & 1 & \\
     & & & & & s & A_{34}^1 \\
     & & & & & & 0
\end{BMAT}
\right] \text{ is contained in } \young(113,2,4) \bdfamo \young(2,3)
\]
results in the following relations on submatrices:
\begin{table}[H]
  \centering
  \newcolumntype{Q}{>{$}l<{$}>{$}l<{$}}
  \begin{tabular}{Q} 
    \text{Submatrix} & \text{Relations} \\
    \midrule 
    A_2 & A_{12}^1 \\
    A_3 & A_{23}^1 +sA_{23}^2 \\
    A_4 & A_{34}^1, A_{13}^1A_{24}^1 - A_{23}^1A_{14}^1 
    \end{tabular}
\end{table}
\noindent At $s = 0$ the ideal generated by the above relations decomposes as a union of the following generalized orbital varieties.
\begin{table}[H]
  \centering
  \newcolumntype{Q}{>{$}l<{$}>{$\small}l<{$}}
  \begin{tabular}{Q} 
    \text{Ideal of } X(\tau) & \tau \\ 
    \midrule 
    (A_{12}^1,A_{13}^1,A_{23}^1,A_{34}^1,s) & \young(1123,24,3) \BS \\
    (A_{12}^1,A_{23}^1,A_{24}^1,A_{34}^1,s) & \young(1123,23,4) \TS
    \end{tabular}
\end{table}
\end{example}

\begin{example}
$\young(2,4) \cdot \young(13) = \young(2,3)\arfu \young(4) + \young(3,4)\arfu \young(2):$ The requirement that
\[
A = \left[\begin{BMAT}(e){c;c;c;c}{c;c;c;c}
     s & A_{12}^1 & A_{13}^1 & A_{14}^1 \\
     & 0 & A_{23}^1 & A_{24}^1 \\
     & & s & A_{34}^1 \\
     & & & 0
\end{BMAT}
\right] \text{ is contained in } \young(2,4) \bdfamo \young(13)
\]
results in the following relations on submatrices:
\begin{table}[H]
  \centering
  \newcolumntype{Q}{>{$}l<{$}>{$}l<{$}}
  \begin{tabular}{Q} 
    \text{Submatrix} & \text{Relations} \\
    \midrule 
     A_4 & A_{23}^1A_{34}^1 - sA_{24}^1
    \end{tabular}
\end{table}
\noindent At $s = 0$ the ideal generated by the above relations decomposes as a union of the following generalized orbital varieties.
\begin{table}[H]
  \centering
  \newcolumntype{Q}{>{$}l<{$}>{$\small}l<{$}}
  \begin{tabular}{Q} 
    \text{Ideal of } X(\tau) & \tau \\ 
    \midrule 
    (A_{23}^1,s) & \young(124,3) \BS \\
    (A_{34}^1,s) & \young(123,4) \TS
    \end{tabular}
\end{table}
\end{example}

\begin{example}

$\young(113,22,34) \cdot \young(1,4) = \young(1,3,4)\arfu\young(4) + \young(1,2,4)\arfu\young(3,4) \,$: The requirement that 
\[
A = \left[\begin{BMAT}(e){ccc;cc;cc;cc}{ccc;cc;cc;cc}
    0 & 1 & & & & & & & \\
     & 0 & 1 & & & & & & \\
     & & s & A_{12}^1 & A_{12}^2 & A_{13}^1 & A_{13}^2 & A_{14}^1 & A_{14}^2 \\
     & & & 0 & 1 & & & & \\
     & & & & 0 & A_{23}^1 & A_{23}^2 & A_{24}^1 & A_{24}^2 \\
     & & & & & 0 & 1 & & \\
     & & & & & & 0 & A_{34}^1 & A_{34}^2 \\
     & & & & & & & 0 & 1 \\
     & & & & & & & & s
\end{BMAT}
\right] \text{ is contained in } \young(113,22,34) \bdfamo \young(1,4)
\]
results in the following relations on submatrices:
\begin{table}[H]
  \centering
  \newcolumntype{Q}{>{$}l<{$}>{$}l<{$}}
  \begin{tabular}{Q} 
    \text{Submatrix} & \text{Relations} \\
    \midrule 
    A_2 & A_{12}^1,A_{12}^2 \\
    A_3 & A_{13}^1, A_{23}^1 \\
    A_4 & A_{34}^1, A_{13}^2 A_{34}^2 + A_{14}^1 + sA_{14}^2, A_{14}^1 A_{23}^2 - A_{13}^2 A_{24}^1
    \end{tabular}
\end{table}
\noindent At $s = 0$ the ideal generated by the above relations decomposes as a union of the following generalized orbital varieties.
\begin{table}[H]
  \centering
  \newcolumntype{Q}{>{$}l<{$}>{$\small}l<{$}}
  \begin{tabular}{Q} 
    \text{Ideal of } X(\tau) & \tau \\ 
    \midrule 
    (A_{12}^1,A_{12}^2,A_{13}^1,A_{13}^2,A_{23}^1,A_{14}^1,A_{34}^1,s) & \young(1114,223,34) \BS \\
     (A_{12}^1,A_{12}^2,A_{13}^1,A_{23}^1,A_{34}^1,A_{23}^2A_{34}^2 + A_{24}^1,   A_{13}^2A_{34}^2+A_{14}^1,A_{13}^2A_{24}^1-A_{23}^2A_{14}^1,s)  & \young(1113,224,34) \TS
    \end{tabular}
\end{table}
\end{example}

\begin{example}

$\young(11,23,34) \cdot \young(2,4) = \young(1,2,4)\arfu\young(3,4) + \young(2,3,4) \arfu \young(1,4) \,$: The requirement that 
\[
A = \left[\begin{BMAT}(e){cc;cc;cc;cc}{cc;cc;cc;cc}
    0 & 1 & & & & & & \\
     & 0 & A_{12}^1 & A_{12}^2 & A_{13}^1 & A_{13}^2 & A_{14}^1 & A_{14}^2 \\
     & & 0 & 1 & & & & \\
     & & & s & A_{23}^1 & A_{23}^2 & A_{24}^1 & A_{24}^2 \\
     & & & & 0 & 1 & & \\
     & & & & & 0 & A_{34}^1 & A_{34}^2 \\
     & & & & & & 0 & 1 \\
     & & & & & & & s
\end{BMAT}
\right] \text{ is contained in } \young(11,23,34) \bdfamo \young(2,4)
\]
results in the following relations on submatrices:
\begin{table}[H]
  \centering
  \newcolumntype{Q}{>{$}l<{$}>{$}l<{$}}
  \begin{tabular}{Q} 
    \text{Submatrix} & \text{Relations} \\
    \midrule 
   A_2 & A_{12}^1 \\
    A_3 & A_{13}^1, A_{23}^1, A_{12}^2A_{23}^2 + sA_{13}^2 \\
    A_4 & A_{34}^1, A_{23}^2 A_{34}^2 + A_{24}^1 + sA_{24}^2, A_{12}^2 A_{24}^2 + A_{13}^2 A_{34}^2 + A_{14}^1, A_{12}^2 A_{24}^1 - sA_{14}^1, \\
    & A_{14}^1 A_{23}^2 - A_{13}^2 A_{24}^1, A_{13}^2 A_{24}^1 A_{34}^2 + A_{14}^1 A_{24}^1 + sA_{14}^1 A_{24}^2
    \end{tabular}
\end{table}
\noindent At $s = 0$ the ideal generated by the above relations decomposes as a union of the following generalized orbital varieties.
\begin{table}[H]
  \centering
  \newcolumntype{Q}{>{$}l<{$}>{$\small}l<{$}}
  \begin{tabular}{Q} 
    \text{Ideal of } X(\tau) & \tau \\ 
    \midrule 
    \begin{array}{c}
     (A_{12}^1, A_{12}^2, A_{13}^1, A_{23}^1, A_{34}^1, A_{23}^2A_{34}^2 + A_{24}^1,  \\
     A_{13}^2A_{34}^2 + A_{14}^1, A_{13}^2A_{24}^1 - A_{23}^2A_{14}^1,s)  
\end{array} & \young(113,224,34) \BS \\
    \begin{array}{c}
      (A_{12}^1, A_{13}^1, A_{23}^1, A_{23}^2, A_{24}^1,  \\
      A_{34}^1, A_{12}^2A_{24}^2 + A_{13}^2A_{34}^2 + A_{14}^1,s) 
 \end{array} & \young(112,234,34) \TS
    \end{tabular}
\end{table}
\end{example}

\begin{example}

$\young(1113,22,34) \cdot \young(2,4) = \young(2,3,4)\arfu\young(4) + \young(1,2,4)\arfu\young(3,4)\arfu\young(2) \,$: The requirement that 
\[
A = \left[\begin{BMAT}(e){ccc;ccc;cc;cc}{ccc;ccc;cc;cc}
    0 & 1 & & & & & & & & \\
     & 0 & 1 & & & & & & & \\
     & & 0 & A_{12}^1 & A_{12}^2 & A_{12}^3 & A_{13}^1 & A_{13}^2 & A_{14}^1 & A_{14}^2 \\
     & & & 0 & 1 & & & & & \\
     & & & & 0 & 1 & & & & \\
     & & & & & s & A_{23}^1 & A_{23}^2 & A_{24}^1 & A_{24}^2 \\
     & & & & & & 0 & 1 & & \\
     & & & & & & & 0 & A_{34}^1 & A_{34}^2 \\
     & & & & & & & & 0 & 1 \\
     & & & & & & & & & s
\end{BMAT}
\right] \text{ is contained in } \young(1113,22,34) \bdfamo \young(2,4)
\]
results in the following relations on submatrices:
\begin{table}[H]
  \centering
  \newcolumntype{Q}{>{$}l<{$}>{$}l<{$}}
  \begin{tabular}{Q} 
    \text{Submatrix} & \text{Relations} \\
    \midrule 
   A_2 & A_{12}^1, A_{12}^2 \\
    A_3 & A_{13}^1,A_{23}^1 \\
    A_4 & A_{34}^1, A_{23}^2 A_{34}^2 + A_{24}^1 + sA_{24}^2, A_{14}^1 A_{23}^2 - A_{13}^2 A_{24}^1, A_{13}^2 A_{24}^1 A_{34}^2 + A_{14}^1 A_{24}^1 + sA_{14}^1 A_{24}^2
    \end{tabular}
\end{table}
\noindent At $s = 0$ the ideal generated by the above relations decomposes as a union of the following generalized orbital varieties.
\begin{table}[H]
  \centering
  \newcolumntype{Q}{>{$}l<{$}>{$\small}l<{$}}
  \begin{tabular}{Q} 
    \text{Ideal of } X(\tau) & \tau \\ 
    \midrule 
    (A_{12}^1, A_{12}^2, A_{13}^1, A_{23}^1, A_{23}^2, A_{24}^1, A_{34}^1,s) & \young(11124,223,34) \BS \\
     (A_{12}^1, A_{12}^2, A_{13}^1, A_{23}^1, A_{34}^1, A_{23}^2A_{34}^2 + A_{24}^1, A_{13}^2A_{34}^2 + A_{14}^1, A_{13}^2A_{24}^1 - A_{23}^2A_{14}^1, s)& \young(11123,224,34) \TS
    \end{tabular}
\end{table}
\end{example}

We note that the tableaux and mutation equations we have obtained match those predicted by \cite{li2020dual} for corresponding dual canonical basis elements.

\bibliographystyle{alpha}
\bibliography{mvybd}

\end{document}